\documentclass[reqno,11pt]{amsart}
\usepackage{amsmath}%
\usepackage{amsfonts}%
\usepackage{amssymb}%
\usepackage{graphicx, color, enumerate}
\usepackage{latexsym}
\usepackage{amscd}
\usepackage{mathrsfs}
\usepackage{cancel}
\usepackage{mathtools}
\usepackage{xcolor}
\usepackage[top=1in, bottom=1.25in, left=1.25in, right=1.25in]{geometry}
\usepackage[pdfborder=000,pdftex=true]{hyperref}

\def\div{\mathop{\rm div}\nolimits}

\numberwithin{equation}{section}

\newtheorem{Theorem}{Theorem}

\newtheorem{Lemma}[Theorem]{Lemma}
\newtheorem{Proposition}[Theorem]{Proposition}
\newtheorem{Definition}[Theorem]{Definition}
\newtheorem{remark}[Theorem]{Remark}
\newtheorem{remarks}[Theorem]{Remarks}

\newenvironment{Remark}{\begin{remark} \rm}{\rule{2mm}{2mm}\end{remark}}

\numberwithin{Theorem}{section}

\def\sideremark#1{\ifvmode\leavevmode\fi\vadjust{\vbox to0pt{\vss
\hbox to0pt{\hskip\hsize\hskip1em
\vbox{\hsize3cm\tiny\raggedright\pretolerance10000
\noindent #1\hfill}\hss}\vbox to8pt{\vfil}\vss}}}

\title[Coupled nonlinear fractional Schr\"odinger equations]{Nonlinear Fractional Schr\"odinger Equations coupled by power--type nonlinearities}

\author{Eduardo Colorado}
\author{Alejandro Ortega}
\email[E. Colorado]{ecolorad@math.uc3m.es}
\email[A. Ortega ]{alortega@math.uc3m.es}
\address[E. Colorado, A. Ortega]{Departamento de Matem\'aticas,
Universidad Carlos III de Madrid, Av. Universidad 30, 28911 Legan\'es (Madrid), Spain}

\thanks{The authors are partially supported by MCIN/AEI/10.13039/501100011033/ under research project PID2019-106122GB-I00. They are also partially supported by the Madrid Government (Comunidad de Madrid-Spain) under the Multiannual Agreement with UC3M in the line of Excellence of University Professors (EPUC3M23), and in the context of the V PRICIT (Regional Programme of Research and Technological Innovation).}

\begin{document}
\maketitle

\begin{abstract}
In this work we study the following class of systems of coupled nonlinear fractional nonlinear Schr\"odinger equations,
\begin{equation*}
\left \{
\begin{array}{l}
(-\Delta)^s u_1+ \lambda_1 u_1= \mu_1 |u_1|^{2p-2}u_1+\beta |u_2|^{p} |u_1|^{p-2}u_1 \quad\text{in }\mathbb{R}^N,\\[3pt]
(-\Delta)^s  u_2 + \lambda_2 u_2= \mu_2 |u_2|^{2p-2}u_2+\beta |u_1|^{p}|u_2|^{p-2}u_2 \quad\text{in }\mathbb{R}^N,
\end{array} \right.
\end{equation*}
where $ u_1,\, u_2\in W^{s,2}(\mathbb{R}^N)$, with $ N=1,\, 2,\, 3$; $\lambda_j,\,\mu_j>0$, $j=1,2$, $\beta\in \mathbb{R}$, $p\geq 2$ and $\displaystyle\frac{p-1}{2p}N<s<1$.
Precisely, we  prove the existence of positive radial bound and ground state solutions provided the parameters $\beta, p, \lambda_j,\mu_j$, ($j=1,\, 2$) satisfy appropriate conditions. We also study the previous system with $m$-equations,
$$
(-\Delta)^s u_j+ \lambda_j u_j =\mu_j |u_j|^{2p-2}u_j+ \sum_{\substack{k=1\\k\neq j}}^m\beta_{jk}  |u_k|^p|u_j|^{p-2}u_j,\quad u_j\in W^{s,2}(\mathbb{R}^N);\: j=1,\ldots,m
$$
where  $\lambda_j,\, \mu_j>0$ for $j=1,\ldots ,m\ge 3$, the coupling parameters $\beta_{jk}=\beta_{kj}\in \mathbb{R}$ for $j,k=1,\ldots,m$, $j\neq k$. For this system we prove similar results as for $m=2$, depending on the values of the parameters $\beta_{jk}, p, \lambda_j,\mu_j$, (for $j,k=1,\ldots,m$, $j\neq k$).
\end{abstract}

\section{Introduction}
In this paper we will study the existence of positive solutions to the following system of coupled nonlinear fractional Schr\"odinger (NLFS) equations involving the fractional Laplace operator:
\begin{equation}\label{eq:sistema-fractional}
\left \{
\begin{array}{l}
(-\Delta)^s u_1+ \lambda_1 u_1= \mu_1 |u_1|^{2p-2}u_1+\beta |u_2|^{p}|u_1|^{p-2}u_1\quad\text{in }\mathbb{R}^N,\\[3pt]
(-\Delta)^s  u_2 + \lambda_2 u_2= \mu_2 |u_2|^{2p-2}u_2+\beta |u_1|^p |u_2|^{p-2}u_2\quad\text{in }\mathbb{R}^N,
\end{array} \right.
\end{equation}
where we assume $u_j\in W^{s,2}(\mathbb{R}^N)$ with $ N=1,\,2,\, 3$; $\lambda_j,\,\mu_j>0$, $j=1,2$, the coupling factor $\beta\in \mathbb{R}$, $p\geq 2$ and $\displaystyle\frac{p-1}{2p}N<s<1$. Note that the condition $\displaystyle\frac{p-1}{2p}N<s<1$ implies $2p<2_s^*$, where $2_s^*= \displaystyle\frac{2N}{N-2s}$ if $s<\displaystyle\frac{N}{2}$ and $2_s^*=\infty $ if $s\geq\displaystyle\frac{N}{2}$, is the critical Sobolev exponent. While condition $p\geq 2$ ensures that the energy functional associated to the system \eqref{eq:sistema-fractional} is a $\mathcal{C}^2$ functional.

We also study the previous system with $m$-equations,
\begin{equation}\label{eq:sys-m}
(-\Delta)^s u_j+ \lambda_j u_j =\mu_j |u_j|^{2p-2}u_j+ \sum_{\substack{k=1\\k\neq j}}^m\beta_{jk}  |u_k|^p|u_j|^{p-2}u_j,\:\: u_j\in W^{s,2}(\mathbb{R}^N);\: j=1,\ldots,m
\end{equation}
where  $\lambda_j,\, \mu_j>0$ for $j=1,\ldots ,m\ge 3$, the coupling parameters $\beta_{jk}=\beta_{kj}\in \mathbb{R}$ for $j,k=1,\ldots,m$, $j\neq k$;
proving similar results as for $m=2$, depending on the values of the parameters $\beta_{jk}, p, \lambda_j,\mu_j$, (for $j,k=1,\ldots,m$, $j\neq k$).

\

Problems like \eqref{eq:sistema-fractional} have been widely investigated with the classical Laplacian ($s=1$) in the last $16$ years. It seems to be very complicated to give a complete list of references (cf. for example \cite{ac1,ac2,BW, BWW,ChenZou,col2,DW,DWW,dfl,linwei,linwei2,LinWu,liu-wang,mmp,pomp,sirakov,TV, WW,WW2} and the references therein).

On the other hand, it is well known that solutions of \eqref{eq:sistema-fractional}, at least for the classical case ($s=1$) are related to the solitary waves of the Gross-Pitaevskii equations, which have applications in many physical models, such as in nonlinear optics (cf. \cite{akanbook,ac2, Ma, Me}) and in Bose-Einstein condensates for multi-species condensates (cf. \cite{CLLL,R}). For example, in optics, a planar light beam propagating in the $z$ direction in a non-linear medium, can be described by a vector  nonlinear Schrödinger equation (NLS) equation like
$$
i\, \textbf{E}_z+ \textbf{E}_{xx}+ \kappa |\textbf{E}|^2 \textbf{E}=0,
$$
where $i$, $\textbf{E}(x,z)$  denote the imaginary unit and  the complex envelope of an Electric field, respectively. Without loss of generality, we can assume $\kappa=1$. If $\textbf{E}$ is the sum of two right- and left-hand polarized waves $a_1E_1$ and $a_2E_2$, $a_j\in\mathbb{R}$, then solitary wave solutions $E_j(z,x)=e^{{\rm i}\lambda_j z}u_j(x)$, where $\lambda_j>0$ and $u_j(x)$ are real valued functions, satisfy the system
\begin{equation}\label{eq:11}
 \left \{
\begin{array}{l}
- (u_1)_{xx}+ \lambda_1 u_1 = (a_1^2 u_1^2+a_2^2 u_2^2)u_1,\\
- (u_2)_{xx} + \lambda_2 u_2 = (a_1^2 u_1^2+a_2^2 u_2^2)u_2.
\end{array} \right.
\end{equation}
If we take the coupling factor $a_2^2:=\beta$ as a parameter and
let the coefficients of $u_j^3$, namely $\mu_j>0$, to be different then
\eqref{eq:11} becomes
\begin{equation*}
 \left \{
\begin{array}{l}
- u''_1+ \lambda_1 u_1 = \mu_1 u_1^3+\beta u_2^2 u_1,\\
- u''_2 + \lambda_2 u_2 = \mu_2 u_2^3+\beta u_1^2u_2.
\end{array} \right.
\end{equation*}
This system corresponds to \eqref{eq:sistema-fractional} with $N=1$, $s=1$ and $p=2$.
When one consider the NLFS equation in $\mathbb{R}^N$,
\begin{equation*}
i\, \textbf{E}_z -(-\Delta)^s\textbf{E}+ \kappa |\textbf{E}|^2 \textbf{E}=0,
\end{equation*}
looking for solitary wave solutions, as in the above discussion, one arrives to system \eqref{eq:sistema-fractional} with $p=2$. Let us also remark that system \eqref{eq:sistema-fractional} can be seen as the fractional analogue to the problem studied in \cite{oliveira}, where similar results for ground states were obtained.

 We point out that this type of nonlocal diffusion operators involving the fractional Laplacian $(-\Delta)^s$, $0<s<1$, arises in several several physical phenomena like flames propagation and
chemical reactions in liquids, population dynamics, geophysical fluid dynamics, also in probability, American option in finance and also in $\alpha$-stable L\'evy processes (with $\alpha=2s$) (see for instance  \cite{Applebaum,Bertoin,Cont-Tankov,VIKH}).

Here we are interested in systems of coupled NLS equations involving the so called fractional Laplacian operator (or fractional Schrödinger operator, $(-\Delta)^s+\lambda\, \mbox{Id}$, (cf. \cite{L,L2})).

\

The main goal of this work is to give a
classification of positive solutions of \eqref{eq:sistema-fractional}
and study the system with $m$-equations \eqref{eq:sys-m}, for which we will show similar existence results. Precisely, we will prove:

\noindent -Existence of positive radial ground states under the
following hypotheses:
\begin{itemize}
\item $p=2$ and the coupling coefficient $\beta>\Lambda'$; see Theorem \ref{th:ac},
\item $p\ge 2$ and the coupling coefficient $\beta$ satisfying hypothesis \eqref{eq:hyp}; see Theorem \ref{th:posground3}.
\end{itemize}
-Existence of radial bound states when:
\begin{itemize}
\item $p=2$ and $\beta<\Lambda$; see Theorem \ref{th:acb}-$(i)$ which are positive provided $\beta>0$,
\item $p>2$ and $\beta\in \mathbb{R}$; see Theorem \ref{th:acb}-$(ii)$, which are positive when $\beta>0$,
\item $p\geq 2$ and $\beta\sim 0$, proving also a bifurcation result; see Theorem \ref{th:acb}-$(iii)$ and Theorem \ref{th:pert}, which are positive for $\beta>0$.
\end{itemize}
 The paper contains 4 more sections. In Section \ref{sec:not-prel} we introduce notation, some back ground on the fractional Laplacian and give the definition of bound and ground states. Section \ref{sec:kr} contains the definition of the Nehari manifold, its properties and some preliminary results based on the Nehari Manifold as a {\it natural constraint} and the key results for getting  the main existence theorems, which are stated and proved in Section \ref{sec:pf2}, divided into two subsections: proving the existence of bound states in the first one, while the second one is devoted to prove the existence of ground states.
 Finally, the aim of Section \ref{sec:ext} is to show some extensions to systems with more than two equations.

\section{Preliminaries and Notation}\label{sec:not-prel}
Given $0<s<1$, the nonlocal operator $(-\Delta)^s$ in $\mathbb{R}^N$ is defined on
the Schwartz class of functions $g\in
\mathcal{S}$ through the Fourier transform,
\begin{equation}\label{fourier}
    [(-\Delta)^s  g]^{\wedge}\,(\xi)=(2\pi|\xi|)^{2s} \,\widehat{g}(\xi),
\end{equation}
or via the Riesz potential, see for example~\cite{Landkof,Stein}. Observe that
$s=1$ corresponds to the standard local Laplacian. See also \cite{brcdps,col,dpv,fl,fls}, where problems dealing with equations involving the fractional Laplacian are studied.

There is another way of defining this operator. In fact, for $s=\frac12$, the square root of the Laplacian acting on a function $u$ in the whole space $\mathbb{R}^N$, can be calculated as the normal derivative on the boundary of its harmonic extension to the upper half-space $\mathbb{R}^{N+1}_+$, this is so-called Dirichlet to Neumann operator. Based on this idea, Caffarelli and Silvestre (cf. \cite{Caffarelli-Silvestre}) proved that this operator can be realized  in a local way by using one more variable and the so called $s$-harmonic extension.

More precisely, given $u$ a regular function in $\mathbb{R}^N$, we define its $s$-harmonic extension to the upper half-space $\mathbb{R}^{N+1}_+$, denoted by $w=E_s[u]$, as the solution to the problem
\begin{equation}\label{extension}
\left\{
\begin{array}{rl}
-\div(y^{1-2s}\nabla w)=0&\mbox{ in } \mathbb{R}^{N+1}_+,\\ [5pt]
w=u&\mbox{ on } \mathbb{R}^N\times\{y=0\}.
\end{array}
\right.
\end{equation}
The key point of the $s$-harmonic extension comes from the following identity
 \begin{equation}
\lim\limits_{y\to0^+}y^{1-2s}\dfrac{\partial w}{\partial
y}(x,y)=-\dfrac{1}{\kappa_s}(-\Delta)^su(x),
\label{normalder}\end{equation}
proved in \cite{Caffarelli-Silvestre}, where $\kappa_s$ is a positive constant.
The above Dirichlet-Neumann procedure
\eqref{extension}--\eqref{normalder} provides a formula for the
fractional Laplacian in $\mathbb{R}^N$, equivalent to that obtained
from Fourier transform, see \eqref{fourier}. In that
case, the $s$-harmonic extension and the fractional Laplacian
have explicit expressions in terms of the Poisson and the
Riesz kernels:
\begin{equation}
  \label{poisson}
\begin{split}
w(x,y)=P_y^s*u(x)&=
c_{N,s}\,y^{2s}\int_{\mathbb{R}^N}\dfrac{u(z)}{(|x-z|^{2}+y^{2})^{\frac{N+2s}{2}}}\,dz\,,
\\
(-\Delta)^su(x)&=
d_{N,s} P.V.\int_{\mathbb{R}^N}\frac{u(x)-u(y)}{|x-y|^{N+2s}}\,dy\,.
\end{split}
\end{equation}
The natural functional spaces are the homogeneous fractional Sobolev space
$\dot H^s(\mathbb{R}^N)$ and the weighted Sobolev space
$X^s(\mathbb{R}^{N+1}_+)$, that can be defined as the completion of $\mathcal{C}_0^{\infty}(\mathbb{R}^N)$ and $\mathcal{C}_{0}^{\infty}(\overline{\mathbb{R}^{N+1}_+})$ respectively, under the norms
\begin{equation}\label{normas}
\begin{array}{l}
\displaystyle\|\psi\|^2_{\dot{H}^s}=\int_{\mathbb{R}^N}|2\pi\xi|^{2s}|\widehat{\psi}(\xi)|^2\,d\xi
=\int_{\mathbb{R}^N}|(-\Delta)^{\frac{s}{2}}\psi(x)|^2\, dx,\\[10pt]
\displaystyle\|\phi\|^2_{X^s}=\kappa_s\int_{\mathbb{R}^{N+1}_+} y^{1-2s}|\nabla
\phi(x,y)|^2\,dx dy,
\end{array}
\end{equation}
where $\kappa_s$ is the constant in \eqref{normalder}. Furthermore, the constants in \eqref{poisson} and
\eqref{normalder} satisfy the identity $2s c_{N,s}\kappa_s=d_{N,s}$. Their explicit
value can be seen  for instance in \cite{brcdps}. Thanks to the choice of the constant $\kappa_s$, the $s$-harmonic extension operator
\begin{equation*}
\begin{split}
E_s: \dot H^s(\mathbb{R}^N)&\to X^s(\mathbb{R}^{N+1}_+)\\
u&\mapsto w=E_s[u]
\end{split}
\end{equation*}
is an isometry between the spaces $\dot H^s(\mathbb{R}^N)$ and $X^s(\mathbb{R}^{N+1}_+)$, i.e., $\displaystyle \|\varphi \|_{\dot H^s}= \| E_s[\varphi]\|_{X^s}$ for all $\varphi\in \dot H^s(\mathbb{R}^N)$. Moreover, there exists a constant $C=C(N,s)>0$ such the following trace inequality holds true (cf. \cite{brcdps}),
\begin{equation*}
\| w(\cdot, 0)\|_{L^{2_s^*}}\le C\|w\|_{X^s}\, , \quad\forall\, w\in X^s(\mathbb{R}^{N+1}_+).
\end{equation*}
Along the work we will use the following notation:
 \begin{itemize}
\item $E:=W^{s,2}(\mathbb{R}^N)$, denotes the fractional Sobolev space endowed with scalar product and norm
\begin{equation*}
(u\mid v)_j=\int_{\mathbb{R}^N} [ (-\Delta)^{\frac{s}{2}} u (-\Delta)^{\frac{s}{2}}  v + \lambda_j uv] dx,\quad \|u\|_j^2=(u\mid u)_j,\:\: j=1,2.
\end{equation*}
\item $\mathbb{E}:=E\times E$; the elements in $\mathbb{E}$ will be denoted by $\textbf{u}=(u_1,u_2)$; as a norm in $\mathbb{E}$ we will take $\|\textbf{u}\|_{\mathbb{E}}^2=\|u_1\|_1^2+\|u_2\|_2^2$.
\item we set $\textbf{0}=(0,0)$;
\item  for $\textbf{u}\in \mathbb{E}$ the notation $\textbf{u}\geq \textbf{0}$ (resp. $\textbf{u}>\textbf{0}$) means that $u_j\geq 0$, (resp. $u_j>0$), for all $j=1,2$.
\end{itemize}
For $u\in E$, (resp. $\textbf{u}\in \mathbb{E}$),  we set
\begin{equation*}
\begin{split}
I_j(u)&= \frac{1}{2} \int_{\mathbb{R}^N} (|(-\Delta)^{\frac{s}{2}}  u|^2+\lambda_j u^2)dx -\frac{1}{2p}\,\mu_j \int_{\mathbb{R}^N}  |u|^{2p}dx,\\
F(\textbf{u})&= F(u_1,u_2)= \frac{1}{2p}\, \int_{\mathbb{R}^N}  \left(\mu_1 |u_1|^{2p} +\mu_2 |u_2|^{2p}\right)dx\\
G(\textbf{u})&= G(u_1,u_2)= \frac{1}{p} \int_{\mathbb{R}^N}  |u_1|^p|u_2|^p dx,\\
\Phi({\bf u}) &=\Phi(u_1,u_2)= I_1(u_1)+I_2(u_2) - \beta\, G(u_1,u_2)\\
&= \frac 12 \|\textbf{u} \|_{\mathbb{E}}^2- F(\textbf{u}) -\beta\,G(\textbf{u}).
\end{split}
\end{equation*}

\

We observe that $F$ and $G$ are well defined because $\frac{p-1}{2p}N<s<1$ guarantees $2p<2_s^*$ which in turn implies the continuous Sobolev embedding $E\hookrightarrow L^{2p}(\mathbb{R}^N)$. Moreover, note that, since $p\geq 2$, the energy functional $\Phi$ is $\mathcal{C}^2$.

Any critical point $\textbf{u}\in \mathbb{E}$ of $\Phi$ gives rise to a solution of system \eqref{eq:sistema-fractional}. If $\textbf{u}\neq \textbf{0}$ we say that such a critical point is non-trivial. We also say that a solution $\textbf{u}$ of \eqref{eq:sistema-fractional} is {\it positive} if $\textbf{u}>\textbf{0}$.

\begin{Definition}\label{def:sol}
We say that ${\bf u}=(u_1,u_2)\in\mathbb{E}$ is a {\it bound state} to \eqref{eq:sistema-fractional} iff it is a critical point of $\Phi$, i.e.,
\begin{equation*}
\begin{split}
(u_1\mid \varphi_1)_1+(u_2\mid \varphi_2)_2 =&  \int_{\mathbb{R}^N} \left(\mu_1|u_1|^{2p-2}u_1\varphi_1
+\mu_2|u_2|^{2p-2}u_2\varphi_2\right)dx\\
&+\beta \int_{\mathbb{R}^N}  \left(|u_1|^{p-2}u_1\varphi_1 |u_2|^{p}+ |u_2|^{p-2}u_2\varphi_2 |u_1|^{p}\right)dx,
\end{split}
\end{equation*}
for any  $(\varphi_1,\varphi_2)\in \mathbb{E}$.
\end{Definition}

Among non-trivial solutions of \eqref{eq:sistema-fractional}, we shall distinguish the {\it bound states} from the {\it ground states}.

\begin{Definition}\label{def:ac}
A positive bound state ${\bf u}>{\bf 0}$  such that its energy is minimal among {\it all the non-trivial bound states}, namely
\begin{equation}\label{eq:gr}
\Phi(\bf{u})=\min\{\Phi({\bf v}): {\bf v}\in \mathbb{E}\setminus\{{\bf 0}\},\; \Phi'({\bf v})=0\},
\end{equation}
is called a {\it ground state}  of \eqref{eq:sistema-fractional}.
\end{Definition}
A relevant fact of the ground states is that they are candidates to be orbitally stable for the corresponding evolution equation, see for instance \cite{cz,stuart}. Nowadays it is well known that a ground state could be also non-negative and nontrivial, i.e., with some of their components zero and at least one positive component. Here we are interested in positive bound and ground states, in the sense of Definitions \ref{def:sol} and \ref{def:ac}.

\section{The Nehari manifold and preliminary results}\label{sec:kr}
In order to find critical points of $\Phi$ it is convenient to introduce the corresponding Nehari manifold. To do so, let us set
\begin{equation*}
\Psi({\bf u}) = (\Phi'({\bf u})\mid {\bf u})=\|{\bf u}\|_{\mathbb{E}}^2 -2p\, F({\bf u}) -2p\beta\, G({\bf u}),
\end{equation*}
then, we define the Nehari manifold as
\begin{equation*}
\mathcal{M} =  \{ {\bf u}\in \mathbb{E}_{rad}\setminus\{{\bf 0}\}: \Psi ({\bf u})=0\},
\end{equation*}
where the subindex $rad$ means radial.

\

Plainly, $\mathcal{M}$ contains all the non-trivial critical points of $\Phi$ on $\mathbb{E}_{rad}$. On the other hand, for any ${\bf v}\in \mathbb{E}_{rad}\setminus \{{\bf 0}\}$, one has that
\begin{equation}\label{eq:tm}
t{\bf v} \in \mathcal{M}\quad \Longleftrightarrow\quad
t^2\|{\bf v}\|_{\mathbb{e}}^2= t^{2p}\left[ 2pF({\bf v})+2p\beta G({\bf v})\right].
\end{equation}
As a consequence, for all ${\bf v}\in \mathbb{E}_{rad}\setminus \{{\bf 0}\}$, there exists
a unique $t>0$ such that $t{\bf v} \in \mathcal{M}$.
Moreover, since $F,G$ are homogeneous with degree $2p$,
that $\exists\,\rho>0$ such that
\begin{equation}\label{eq:M1}
\|{\bf u}\|_{\mathbb{E}}^2\geq \rho,\qquad \forall\, {\bf u}\in \mathcal{M}.
\end{equation}
Furthermore, from \eqref{eq:M1} it follows that
 \begin{equation}\label{eq:M2}
 (\Psi'({\bf u})\mid {\bf u})=(2-2p)\|{\bf u}\|_\mathbb{E}^2\le 2(1-p)\rho< 0,\qquad \forall\,{\bf u}\in \mathcal{M}.
\end{equation}
From \eqref{eq:M1} and \eqref{eq:M2} it follows that the Nehari manifold $\mathcal{M}$ is a smooth complete manifold of codimension $1$ in $\mathbb{E}_{rad}$.  Moreover, if
 ${\bf u}\in \mathcal{M}$ is a critical point of $\Phi$ constrained to $\mathcal{M}$, then
 there exists $\omega\in \mathbb{R}$ such that
\begin{equation*}
 \Phi'({\bf u})= \omega \Psi'({\bf u}).
\end{equation*}
Hence one finds $\Psi({\bf u})=( \Phi'({\bf u})\mid {\bf u})=\omega (\Psi'({\bf u})\mid {\bf u})$. Since
 $\Psi({\bf u})=0$, while by \eqref{eq:M2} $(\Psi'({\bf u})\mid {\bf u})<0$, we infer that
 $\omega=0$ and thus  $ \Phi'({\bf u})= 0$. In conclusion, we can state the following Proposition.
 \begin{Proposition}\label{pr:ac}
 We have that ${\bf u} \in \mathbb{E}$ is a non-trivial critical point of $\Phi$ if and only if
 ${\bf u}\in \mathcal{M}$ and is a critical point of $\Phi$ constrained to $\mathcal{M}$.
 \end{Proposition}
 \begin{Remark}\label{rem:morse}
 The Morse index of a critical point  of $\Phi$ is the maximal dimension of the subspace where $\Phi''$ is negatively defined. From \eqref{eq:M2} we deduce that, if ${\bf u}\in\mathcal{M}$, then $\Phi''({\bf u})[{\bf u}]<0$, so that the Morse index of a critical point ${\bf u}_0$ of $\Phi$ is equal to its Morse index as critical point of $\Phi$ constrained to $\mathcal{M}$ plus 1.
 \end{Remark}

 Because of Proposition \ref{pr:ac}, $\mathcal{M}$ is called a {\it natural constraint} for $\Phi$.
 The key point for working on the Nehari manifold is that
 $\Phi$ is bounded from below on $\mathcal{M}$ and, hence, one can try to minimize $\Phi$ on $\mathcal{M}$. Actually, from $\Psi({\bf u})=0$ and the definition of $\mathcal{M}$, it follows that
\begin{equation}\label{eq:F}
\|{\bf u} \|_\mathbb{E}^2= 2p F({\bf u}) +2p\beta G({\bf u}).
\end{equation}
Substituting into $\Phi$ we get
\begin{equation}\label{eq:M3}
\Phi({\bf u})=\frac{p-1}{2p} \|{\bf u} \|_\mathbb{E}^2,\qquad \forall\,{\bf u}\in \mathcal{M},
\end{equation}
or, equivalently,
\begin{equation}\label{eq:M4}
\Phi({\bf u})=(p-1) \left[ F({\bf u}) +\beta G({\bf u})\right],\qquad \forall\,{\bf u}\in \mathcal{M}.
\end{equation}
Then \eqref{eq:M3} jointly with \eqref{eq:M1} imply that there exists a constant $C>0$ such that
\begin{equation}\label{eq:bound}
\Phi({\bf u})\geq C>0,\qquad  \forall\,{\bf u}\in \mathcal{M}.
\end{equation}
Concerning the Palais-Smale (PS) condition, we remember that, in the one dimensional case, $N=1$, we cannot expect a compact embedding of $E$ into $L^q(\mathbb{R})$ for any $1< q<2_s^*$. Indeed, in the even case it is also not true.
Nevertheless, we will prove (see Proposition \ref{prop:acex} below) that for a given PS sequence we can find a subsequence for which the weak limit is a bound state. On the other hand, for the radial case, due to the compact embeddings, for $1<N\leq3$ the PS condition follows by a standard argument as the following Lemma shows.

\begin{Lemma}\label{lem:psN}
Assume $1<N\leq3$. Then $\Phi$ satisfies the (PS) condition on $\mathcal{M}$, namely,
\begin{center}
(PS) every ${\bf u}_n\in\mathcal{M}$ such that $\Phi({\bf u}_n)\to c$ and $\nabla_{\mathcal{M}}\Phi({\bf u}_n)\to 0$ has a strongly convergent subsequence: $\exists\, {\bf u}_0\in\mathcal{M}$ such that ${\bf u}_n\to{\bf u}_0$.
\end{center}
\end{Lemma}
\begin{proof}
Let ${\bf u}_n\in\mathcal{M}$ be a sequence such that $\Phi({\bf u}_n)\to c$. Note that, from \eqref{eq:M1} and \eqref{eq:M3}, we have $c>0$. Moreover, $\Phi({\bf u}_n)\to c$ jointly with \eqref{eq:M3} also implies that ${\bf u}_n$ is bounded and, up to a subsequence, we can assume ${\bf u}_n\rightharpoonup {\bf u}_0$. Since for $N=2,3$ the space $E_{rad}$, is compactly embedded into $L^{2p}(\mathbb{R}^N)$. Thus, we deduce that $F({\bf u}_n)+\beta G({\bf u}_n)\to F({\bf u}_0)+\beta G({\bf u}_0)$. Moreover, by \eqref{eq:M1} and \eqref{eq:F}, there exists $c_0>0$ such that $F({\bf u}_0)+\beta G({\bf u}_0)\geq c_0$ and, hence, ${\bf u}_0\neq {\bf 0}$. Let $\nabla_{\mathcal{M}}\Phi({\bf u})=\Phi'({\bf u})-\omega\Psi({\bf u})$ be the constrained gradient of $\Phi$ to $\mathcal{M}$ and assume that $\nabla_{\mathcal{M}}\Phi({\bf u}_n)\to0$. Taking the dual product with ${\bf u}_n$ and recalling that $( \Phi'({\bf u}_n)\mid {\bf u}_n)=\Psi({\bf u}_n)=0$, we find that $\omega_n( \Psi'({\bf u}_n)\mid {\bf u}_n)\to0$ that, jointly with \eqref{eq:M2}, implies $\omega_n\to0$. Since, in addition, $\|\Psi'({\bf u}_n)\|\leq c_1<+\infty$ we deduce that $\Phi'({\bf u}_n)\to 0$. We conclude by proving that ${\bf u}_n\to{\bf u}_0$ strongly. Indeed, from $\Phi'({\bf u}_n)={\bf u}_n-(F'({\bf u}_n)+\beta G'({\bf u}_n))$, $\Psi({\bf u}_n)=2{\bf u}_n-2p(F'({\bf u}_n)+\beta G'({\bf u}_n))$ and $\Phi'({\bf u}_n)-\omega_n\Psi'({\bf u}_n)=o(1)$ we get
\begin{equation*}
(1-2\omega_n){\bf u}_n=(1-2p\omega_n)(F'({\bf u}_n)+\beta G'({\bf u}_n))+o(1).
\end{equation*}
Next, observe that $F'$ and $G'$ are compact operators since they are the gradients of smooth weakly continuous functionals $F$ and $G$, respectively. This, the preceding equation and $\omega_n\to0$ imply that ${\bf u}_n\to F'({\bf u}_0)+\beta G'({\bf u}_0)$ strongly. Then,  ${\bf u}_0 = F'({\bf u}_0) + \beta G'({\bf u}_0)$, so that $\|{\bf u}_0\|_{\mathbb{E}}^2 = 2p(F ({\bf u}_0) + \beta G({\bf u}_0))$, and we conclude that ${\bf u}_0\in\mathcal{M}$.
\end{proof}

\

It has been proved that the problem
\begin{equation}\label{eq:soliton-alpha}
(-\Delta)^s u +u= |u|^{\alpha}u \quad\mbox{in }\mathbb{R}^N,\quad u\in E, \quad u\not\equiv 0,
\end{equation}
has a unique radial and positive solution (cf. \cite{fl,fls}) for $0<\alpha<\frac{4s}{N-2s}$. On the other hand, it is clear that, for every $\beta\in \mathbb{R}$, \eqref{eq:sistema-fractional} already has
two semi-trivial positive solutions,
\begin{equation*}
{\bf u}_1 = (U_1,0),\quad {\bf u}_2=(0,U_2),
\end{equation*}
 where $U_j$ is the unique radial positive solution of
 \begin{equation}\label{eq:uncoupled}
 (-\Delta)^s u+\lambda_ju=\mu_j |u|^{2p-2}u.
 \end{equation}
Therefore, if $U$ denotes the unique positive radial solution of \eqref{eq:soliton-alpha}, and we set
\begin{equation}\label{eq:soluncoupled}
U_j(x)=\left(\frac{\lambda_j}{\mu_j}\right)^{\frac{1}{2p-2}}\,U(\lambda_j^{\frac{1}{2s}}\,x),\quad j=1,2,
\end{equation}
it follows that $U_j$ are solutions of \eqref{eq:uncoupled}. Then, to find a non-trivial existence result, one has to find solutions having {\it both the components} not identically zero.

\

We are ready to show that there exist non-negative solutions of  \eqref{eq:sistema-fractional}  different from
${\bf u}_j$, $j=1,2$. First, we define
\begin{equation*}
\gamma_1^2 =  \inf_{\varphi\in E\setminus\{0\}} \frac{ \|\varphi\|_2^2 }{\displaystyle\int_{\mathbb{R}^N}U_1^2\varphi^2dx},\qquad
\gamma_2^2 =  \inf_{\varphi\in E\setminus\{0\}} \frac{\|\varphi\|_1^2}{\displaystyle\int_{\mathbb{R}^N} U_2^2\varphi^2dx},
\end{equation*}
and
\begin{equation*}
\Lambda = \min \{\gamma_1^2,\gamma_2^2\},\qquad \Lambda'=\max \{\gamma_1^2,\gamma_2^2\}.
\end{equation*}

\noindent As next Proposition shows, if $p=2$, the semi-trivial solutions ${\bf u}_j$, $j=1,2$, are minimums (resp. saddle points) of $\Phi$ constrained to $\mathcal{M}$, provided $\beta<\Lambda$ (resp. $\beta>\Lambda'$). On the other hand, for $p>2$, the semi-trivial solutions ${\bf u}_j$, $j=1,2$, are both minimums of $\Phi$ constrained to $\mathcal{M}$ for any $\beta\in\mathbb{R}$.
\begin{Proposition}\label{pr:ac1}The following holds:
\begin{enumerate}
\item If $p=2$, then
\begin{itemize}
\item[$(i)$] for any $\beta<\Lambda$, the semi-trivial solutions ${\bf u}_j$, $j=1,2$, are strict local minima of $\Phi$ constrained to $\mathcal{M}$.
\item[$(ii)$] for any $\beta>\Lambda'$, the semi-trivial solutions ${\bf u}_j$, $j=1,2$, are saddle points of $\Phi$ constrained to $\mathcal{M}$. In particular, we have  $\inf_{\mathcal{M}}\Phi < \min\{\Phi({\bf u}_1),\Phi({\bf u}_2)\}$.
\end{itemize}
\item If $p>2$, for any $\beta\in\mathbb{R}$, the semi-trivial solutions ${\bf u}_j$, $j=1,2$, are strict local minima of $\Phi$ constrained to $\mathcal{M}$.
\end{enumerate}
\end{Proposition}
The proof of Proposition \ref{pr:ac1} consists on the evaluation of the Morse index of ${\bf u}_j$, as critical points of $\Phi$ constrained to $\mathcal{M}$. To that end we observe that if $D^2\Phi_{\mathcal{M}}$ denotes the second derivative of $\Phi$ constrained to $\mathcal{M}$, since $\Phi'({\bf u}_j)=0$ then
$D^2\Phi_{\mathcal{M}} ({\bf u}_j)[{\bf h}]^2=\Phi'' ({\bf u}_j)[{\bf h}]^2$ for any ${\bf h}\in T_{{\bf u}_j}\mathcal{M}$. We define the Nehari manifolds associated to $I_j$.
\begin{equation*}
\mathcal{N}_j =\big\{u\in E_{rad} : (I_j'(u)|u)_j=0\big\}=\left\{u\in E_{rad} : \|u\|_j^2 -\mu_j \int_{\mathbb{R}^N} |u|^{2p}=0\right\}.
\end{equation*}
Since $I'_j(U_j)=0$, then $D^2 (I_j)_{\mathcal{N}_j} (U_j)[h]^2=I_j'' ({\bf u}_j)[h]^2$ for any $h\in T_{U_j}\mathcal{N}_j$. Moreover it is easy to check that, for $j=1,\,2$, we have
\begin{equation}\label{eq:tang}
{\bf h}\in T_{{\bf u}_j}\mathcal{M}\quad\text{iff}\quad h_j\in T_{U_j}\mathcal{N}_j.
\end{equation}
\begin{proof}[Proof of Proposition \ref{pr:ac1}]
First let us note that, for ${\bf h}\in\mathcal{M}$, we have
\begin{equation*}
\begin{split}
\Phi''({\bf u})[{\bf h}]^2=&\, I_1''(u_1)[ h_1]^2+I_2''(u_2)[h_2]^2\\
&-\beta p(p-1)\int_{\mathbb{R}^N}|u_1|^{p-2}|u_2|^ph_1^2dx-\beta p(p-1)\int_{\mathbb{R}^N}|u_1|^p|u_2|^{p-2}h_2^2dx\\
&-2\beta p^2\int_{\mathbb{R}^N}|u_1|^{p-2}u_1|u_2|^{p-2}u_2h_1h_2dx
\end{split}
\end{equation*}
Next, we distinguish between the cases $p=2$ and $p>2$. In the first case, $p=2$, we have the following.

\begin{itemize}
\item[$(i)$] Let us assume $\beta<\Lambda$. For any ${\bf h}=(h_1,h_2)\in T_{{\bf u}_1}\mathcal{M}$ it follows that
\begin{equation*}
\Phi''({\bf u}_1)[{\bf h} ]^2 =I_1''(U_1)[ h_1]^2 + \| h_2\|_2^2 -2\beta\int U_1^2h_2^2.
\end{equation*}
Since $U_1$ is a minimum of $I_1$ on $\mathcal{N}_1$, there exists $c_1>0$ such that
\begin{equation*}
I_1''(U_1)[h]^2 \geq c_1 \|h\|_1^2,\qquad \forall \,h\in T_{U_1} \mathcal{N}_1.
\end{equation*}
Then, taking ${\bf h}\in T_{{\bf u}_1}\mathcal{M}$, i.e., $h_1\in T_{U_1}\mathcal{N}_1$,  we get
\begin{equation*}
\Phi''({\bf u}_1)[{\bf h}]^2\geq c_1 \|h_1\|_1^2+\|h_2\|_2^2-\beta \int U_1^2h_2^2
\geq c_1 \|h_1\|_1^2+\left(1-\frac{\beta}{\gamma_1^2}\right)\|h_2\|_2^2.
\end{equation*}
Therefore, since $\beta<\gamma_1^2$ there exists $c_2>0$ such that
\begin{equation*}
\Phi''({\bf u}_1)[{\bf h}]^2\geq c_1 \|h_1\|_1^2+c_2\|h_2\|_2^2>0\quad\forall\,{\bf h}\in T_{{\bf u}_1}\mathcal{M}.
\end{equation*}
Similarly,  since $\beta<\gamma_2^2$,  there exists $c'_i>0$ such that
\begin{equation*}
\Phi''({\bf u}_2)[{\bf h}]^2\geq c'_1 \|h_1\|_1^2+c'_2\|h_2\|_2^2>0\quad\forall\,{\bf h}\in T_{{\bf u}_2}\mathcal{M}.
\end{equation*}
Then, we deduce that both ${\bf u}_1$ and ${\bf u}_2$ are local strict minimums of $\Phi$ on $\mathcal{M}$.

\item[$(ii)$] Let us assume $\beta>\Lambda'$. To prove the result we will evaluate $\Phi''({\bf u}_1)$ on tangent vectors of the form ${\bf h}=(0,h_2)$. First, note that
\begin{equation*}
\Phi''({\bf u}_1)[(0,h_2)]^2 = \|h_2\|_2^2 -\beta \int_{\mathbb{R}^N} U_1^2h_2^2dx.
\end{equation*}
Moreover, by \eqref{eq:tang}, we have that $(0,h_2)\in T_{{\bf u}_1}\mathcal{M}$ for all $h_2\in E_{rad}$. On the other hand, since $\beta>\gamma_1^2$ then there exists $\widetilde{h}_2\in E_{rad}\setminus\{0\}$ such that
\begin{equation*}
\gamma_1^2 < \frac{\|\widetilde{h}_2\|_2^2}{\displaystyle\int_{\mathbb{R}^N} U_1^2\widetilde{h}_2^2dx}<\beta,
\end{equation*}
and, hence,
\begin{equation*}
\Phi''({\bf u}_1)[(0,\widetilde{h}_2)]^2 = \|\widetilde{h}_2\|_2^2 -\beta \int_{\mathbb{R}^N} U_1^2\widetilde{h}_2^2dx<0.
\end{equation*}
Similarly, since $\beta>\gamma_2^2$, there exist $\widetilde{h}_1$ such that $\Phi''({\bf u}_2)[(\widetilde{h}_1,0)]^2<0$.
\end{itemize}
The case $p>2$ follows easily since, in this case, we have that
\begin{itemize}
\item for any ${\bf h}=(h_1,h_2)\in T_{{\bf u}_1}\mathcal{M}$,
\begin{equation*}
\Phi''({\bf u}_1)[{\bf h} ]^2 = I_1''(U_1)[h_1]^2 + \| h_2\|_2^2\geq c_1 \|h_1\|_1^2+\|h_2\|_2^2\geq c\|{\bf h}\|^2>0.
\end{equation*}

\item for any ${\bf h}=(h_1,h_2)\in T_{{\bf u}_2}\mathcal{M}$
\begin{equation*}
\Phi''({\bf u}_2)[{\bf h} ]^2 = \| h_1\|_1^2+I_2''(U_2)[h_2]^2 \geq  \|h_1\|_1^2+ c_2\|h_2\|_2^2\geq c\|{\bf h}\|^2>0.
\end{equation*}
\end{itemize}
\end{proof}

\begin{Remark}
Taking in mind Remark \ref{rem:morse}, by Proposition \ref{pr:ac1} we have that, if $p=2$ and $\beta<\Lambda$ or $p>2$ and $\beta\in\mathbb{R}$, the semi-trivial solutions $u_j$, $j = 1, 2$ have Morse index 1, while for $p=2$ and $\beta>\Lambda$ they have Morse index greater than 1.
\end{Remark}

We finish this section recalling a measure theoretic lemma that will be useful in the sequel (cf. \cite{Lions} or \cite[Lemma 3.6]{col} for a proof).
\begin{Lemma}\label{lem:measure}
For any $2<q<2_s^*$, there exists a positive constant $C=C(q)$ such that
\begin{equation}\label{eq:measure}
\int_{\mathbb{R}} |u|^q\le C\left( \sup_{z\in\mathbb{R}}\int_{|x-z|\le 1}|u(x)|^2dx\right)^{\frac{q-2}{2}}
\| u\|^2_{E},\quad \forall\: u\in E.
\end{equation}
\end{Lemma}
\section{Existence Results}\label{sec:pf2}
Because of Proposition \ref{pr:ac}, in order to find a non-trivial solution of \eqref{eq:sistema-fractional} it is enough to find a critical point of $\Phi$ constrained to $\mathcal{M}$. This is accomplished using Proposition \ref{pr:ac1} and, due to the lack of a compact embedding in the one dimensional case, by proving that for a given PS sequence we can find a subsequence for which the weak limit is a bound state. For dimensions $N=2,3$ this is a direct consequence of Proposition \ref{pr:ac1} and Lemma \ref{lem:psN}.

\begin{Proposition}
\label{prop:acex} The following holds:
\begin{enumerate}
\item If $p=2$,

\begin{itemize}
\item[$(i)$] for any $\beta<\Lambda$, the functional $\Phi$ has a Mountain-Pass (MP) critical point  ${\bf u}^*$ on $\mathcal{M}$. Moreover, one has $\Phi({\bf u}^*)>\max\{\Phi({\bf u}_1) ,\Phi({\bf u}_2)\}$.

\item[$(ii)$] for any $\beta>\Lambda'$, the functional $\Phi$ has a global minimum $\widetilde{{\bf u}}$ on $\mathcal{M}$. Moreover, one has $\Phi(\widetilde{{\bf u}})<\min\{\Phi({\bf u}_1) ,\Phi({\bf u}_2)\}$.
\end{itemize}
\item If $p>2$, for any $\beta\in\mathbb{R}$ the functional $\Phi$ has a MP critical point ${\bf u}^{*}$ on $\mathcal{M}$. Moreover, one has $\Phi({\bf u}^{*})>\max\{\Phi({\bf u}_1) ,\Phi({\bf u}_2)\}$.
\end{enumerate}
\end{Proposition}

\begin{proof} To prove the result we distinguish between the case $N=1$, where no compact embedding holds, and the cases $N=2,3$ where we have a Sobolev compact embedding.

\
\begin{enumerate}
\item Let us start assuming $N=1$ and $p=2$.

\begin{itemize}
\item[$(i)$] By Proposition \ref{pr:ac1}--$\textit{(1)}$-$(i)$, the semi-trivial solutions ${\bf u}_j$, $j=1,\, 2$  are local minima of $\Phi$ constrained to $\mathcal{M}$. This fact allow us to apply the MP theorem (cf. \cite{ar}) to $\Phi$ on $\mathcal{M}$, yielding a (PS) sequence $\{{\bf v}_k\}\subset\mathcal{M}$ with $\Phi ({\bf v}_k)\to c$ and $\nabla_{\mathcal{M}}\Phi({\bf v}_k)\to 0$ as $k\to\infty$, where
\begin{equation*}
c=\inf\limits_{\gamma\in\Gamma}\max\limits_{t\in[0,1]}\Phi(\gamma(t)),
\end{equation*}
and $\Gamma=\{\gamma:[0,1]\to\mathcal{M}\ \text{continuous}\ |\ \gamma(0)={\bf u}_1,\ \gamma(1)={\bf u}_2\}$. Moreover, denoting by $\nabla_{\mathcal{M}}\Phi({\bf u})=\Phi'({\bf u})-\omega \Psi'({\bf u})$, with  $\omega\in \mathbb{R}$, the constrained gradient of $\Phi$ on $\mathcal{M}$, we can suppose that
 $\nabla_{\mathcal{M}}\Phi({\bf v}_k)\to 0$.
Note that by \eqref{eq:M3} one finds that $\{{\bf v}_k\}$ is a bounded sequence on $\mathbb{E}$ so that, up to a subsequence, we can assume that ${\bf v}_k\rightharpoonup {\bf u}^*$ weakly in $\mathbb{E}$, ${\bf v}_k\to {\bf v}^*$ strongly in $L^q_{loc}(\mathbb{R})\times L^q_{loc}(\mathbb{R})$ and a.e. in $\mathbb{R}$.
Following the proof of Lemma \ref{lem:psN}, we deduce that $\Phi'({\bf v}_k)\to 0$.

Next, we define $\rho_k=(v_1)_k^2+(v_2)_k^2$ where ${\bf v}_k=((v_1)_k,(v_2)_k)$ and we prove that there exist $R,\, C>0$ such that
\begin{equation}\label{eq:vanish}
\sup_{z\in\mathbb{R}}\int_{|z-x|<R}\rho_k(x)dx\ge C>0,\quad\forall k\in\mathbb{N},
\end{equation}
that is, there is no evanescence effect. On the contrary, if
\begin{equation*}
\sup_{z\in\mathbb{R}}\int_{|z-x|<R}\rho_k(x)dx\to 0,
\end{equation*}
because of Lemma \ref{lem:measure}  we find that ${\bf v}_k\to {\bf 0}$ strongly in
$L^{2p}(\mathbb{R})\times L^{2p}(\mathbb{R})$, so the weak limit ${\bf v}^*\equiv {\bf 0}$. Since ${\bf v}_k\in\mathcal{M}$, by \eqref{eq:M3} we have
\begin{equation*}
c+o_k(1)=\Phi({\bf v}_k)=\frac{(p-1)}{2p}\|{\bf v}_k\|_{\mathbb{E}}^2,
\end{equation*}
with $o_k(1)\to0$ as $k\to\infty$. This a contradiction with the fact ${\bf v}^*={\bf 0}$, hence there is no evanescence and \eqref{eq:vanish} is proved. Next, we observe that we can find a bounded sequence
$\{z_k\}\subset\mathbb{R}$ such that the new sequence $\overline{\rho}_k(x)=\rho_k(x+z_k)$ verifies
\begin{equation*}
\liminf_{n\to\infty}\int_{|x|<R}\overline{\rho}_n(x)dx\ge \eta >0.
\end{equation*}
Taking into account that $\overline{\rho}_k\to \overline{\rho}$ strongly in $L_{loc}^1(\mathbb{R})$,
we conclude that $\overline{\rho}\not\equiv {\bf 0}$. Therefore, if we define $\overline{{\bf v}}_k(x)={\bf v}_k(x+z_k)$, the sequence $\overline{{\bf v}}_k$ is also a PS sequence for $\Phi$, which is not in $\mathcal{M}$ because it is radial but not with respect to the origin,  the weak limit of $\overline{{\bf v}}_k$, denoted by ${\bf u}^*$ is a critical point of $\Phi$ that, after a translation, is in $\mathcal{M}$, hence it is a critical point of $\Phi$ constrained to $\mathcal{M}$ so that ${\bf u}^*\in\mathcal{M}$. Hence, using \eqref{eq:M4}, and Fatou's Lemma, we get
\begin{equation*}
\begin{split}
\Phi ({\bf u}^*) = &\, (p-1)[ F({\bf u}^*)+\beta G({\bf u}^*)]\\
\ge&\limsup_{k\to\infty}\,(p-1)[ F(\overline{{\bf v}}_k)+\beta G(\overline{{\bf v}}_k)]\\
  = & \limsup_{k\to\infty}\Phi(\overline{{\bf v}}_k)= c.
 \end{split}
\end{equation*}
Similarly, substituting the $\limsup$ by $\liminf$ in the above expression we obtain the reverse inequality, and hence $\Phi({\bf u}^*)=c$. Now it is clear  that ${\bf u}^*$ is a  non-trivial radial
 {\it bound state} solution of \eqref{eq:sistema-fractional}, and moreover, by the MP theorem, we find that $\Phi({\bf u}^*)>\max\{\Phi({\bf u}_1) ,\Phi({\bf u}_2)\}$.
\item[$(ii)$] Arguing in a similar way as before, one proves that  $\inf_{\mathcal{M}}\Phi$ is achieved at some $\widetilde{{\bf u}}\in\mathcal{M}$. Moreover, if $\beta>\Lambda'$,  Proposition \ref{pr:ac1}--$\textit{(1)}$-$(ii)$ implies that $\Phi(\widetilde{{\bf u}})<\min\{\Phi({\bf u}_1) ,\Phi({\bf u}_2)\}$.
\end{itemize}
Since, by Proposition \ref{pr:ac1}-$\textit{(2)}$, the semi-trivial solutions ${\bf u}_j$, $j=1,2$, are strict local minima of $\Phi$ constrained to $\mathcal{M}$, the result for $N=1$ and $p>2$ follows as in the above case $(1)$-$(i)$.

\

\item To finish, we assume $N=2,3$. Then, if $p=2$,
\begin{itemize}
\item[$(i)$] by Lemma \ref{lem:psN} and Proposition \ref{pr:ac1}-$(i)$ we can apply the Mountain-Pass theorem to $\Phi$ on $\mathcal{M}$ yielding a critical point ${\bf u}^*$ of $\Phi$ constrained to $\mathcal{M}$. By the Mountain-Pass theorem it also follows that $\Phi({\bf u}^*)>\max\{\Phi({\bf u}_1) ,\Phi({\bf u}_2)\}$.

\item[$(ii)$] by Lemma \ref{lem:psN} the $\inf_{\mathcal{M}}\Phi$ is achieved at some $\widetilde{{\bf u}}\neq{\bf 0}$. Moreover, if $\beta>\Lambda'$, then Proposition \ref{pr:ac1}-$\textit{(1)}$-$(ii)$ implies that $\Phi(\widetilde{{\bf u}})<\min\{\Phi({\bf u}_1) ,\Phi({\bf u}_2)\}$.
\end{itemize}
Finally, if $p>2$, by Proposition \ref{pr:ac1}-$\textit{(2)}$, the semi-trivial solutions ${\bf u}_j$ $j=1,2$, are strict local minimums of $\Phi$ constrained to $\mathcal{M}$, and the result follows arguing as in the former case  $(2)$-$(i)$.
\end{enumerate}
\end{proof}

\subsection{Existence of ground states}\label{subsec:gr}

\

Concerning ground states, our main result is the following Theorem \ref{th:ac}. Roughly speaking, from Proposition \ref{prop:acex} --$\textit{(1)}$-$(ii)$, only for $p=2$ and $\beta>\Lambda'$ we deduce the existence of a MP critical point $\widetilde{{\bf u}}$ such that its energy $\Phi(\widetilde{{\bf u}})<\min\{\Phi({\bf u}_1) ,\Phi({\bf u}_2)\}$, which will lead to a ground state. In the remaining cases of Proposition \ref{prop:acex}, namely $p=2$ and $\beta<\Lambda$ or $p>2$ and $\beta\in\mathbb{R}$, the conclusion of Proposition \ref{prop:acex} provides us with MP critical points whose energy is strictly greater than $\max\{\Phi({\bf u}_1) ,\Phi({\bf u}_2)\}$, which will not lead a ground state in the sense of Definition \ref{def:ac}.
\begin{Theorem}\label{th:ac}
Assuming $p=2$ and $\beta>\Lambda'$ the system \eqref{eq:sistema-fractional}
has a positive radial ground state $\widetilde{{\bf u}}$.
\end{Theorem}
\begin{proof}
Proposition \ref{prop:acex}--$\textit{(1)}$-$(ii)$ yields a critical point $\widetilde{{\bf u}}\in \mathcal{M}$
which is a non-trivial solution of \eqref{eq:sistema-fractional}.
To finish, we have to
show that $\widetilde{{\bf u}}>{\bf 0}$ and it is a ground state in the sense
of Definition \ref{def:ac}. To prove these facts, we argue as follows.
First we note that the thanks to the Stroock-Varopoulos inequality (cf. \cite{stroock,varopoulos}),
\begin{equation*}
\left\|(-\Delta)^s(|u|) \right\|_{L^2}\leq \left\|(-\Delta)^su \right\|_{L^2},
\end{equation*}
it is easy to show that 
\begin{equation}\label{eq:abs}
|\widetilde{{\bf u}}|=(|\widetilde{u}_1|,|\widetilde{u}_2|)\in\mathcal{M}\quad\text{and}\quad\Phi(|\widetilde{{\bf u}}|)=\Phi(\widetilde{{\bf u}})= \min \{\Phi({\bf u}):{\bf u}\in \mathcal{M}\},
\end{equation}
and, hence, we can assume that $\widetilde{{\bf u}}\geq{\bf 0}$. If one of the components of $\widetilde{{\bf u}}$ would be zero, say $\widetilde{u}_2=0$, then $\widetilde{u}_1\geq 0$, $\widetilde{u}_1\neq 0$ would be solution to the problem
\begin{equation*}
(-\Delta)^s u+\lambda_1u=\mu_1 u^{2p-1} \quad \mbox{in } \mathbb{R}^{N},
 \end{equation*}
As a consequence, $\widetilde{u}_1=U_1$ (see \eqref{eq:soluncoupled}) which is a contradiction with $\Phi(\widetilde{{\bf u}})<\Phi({\bf u}_1)$. Then both components of $\widetilde{{\bf u}}$ satisfy $\widetilde{u}_j\geq0$, $\widetilde{u}_j\neq0$, $j=1,2$. Moreover, $\widetilde{u}_j$ satisfies the problem
\begin{equation*}
(-\Delta)^s \widetilde{u}_j+\lambda_j\widetilde{u}_j=\mu_j \widetilde{u}_j^{2p-1} +\beta \widetilde{u}_i^p\widetilde{u}_j^{p-1}\quad \mbox{in } \mathbb{R}^{N},
\end{equation*}
with $i\neq j$,  $i,j=1,2$. Hence, by the strong maximum principle
we conclude that $\widetilde{u}_j>0$, $j=1,2$, in $\mathbb{R}^{N}$, that is $\widetilde{{\bf u}}>{\bf 0}$.
Now, it remains to prove that
\begin{equation}\label{eq:gs}
\Phi(\widetilde{{\bf u}})= \min \{\Phi({\bf u}):{\bf u}\in \mathbb{E}\setminus\{{\bf 0}\},\, \Phi'({\bf u})=0\}.
\end{equation}
By contradiction,
 let $\widetilde{{\bf v}}\in \mathbb{E}$ be a non-trivial critical point of $\Phi$ such that
\begin{equation}\label{eq:sym}
\Phi(\widetilde{{\bf v}})<\Phi(\widetilde{{\bf u}})= \min \{ \Phi({\bf u}):\: {\bf u}\in \mathcal{M}\}.
\end{equation}
Setting ${\bf u}_0 =|\widetilde{{\bf v}}|$ there holds
\begin{equation}\label{eq:sym1}
\Phi({\bf u}_0)=\Phi(\widetilde{{\bf v}}),\qquad \Psi({\bf u}_0)=\Psi(\widetilde{{\bf v}}).
\end{equation}
Denote by ${\bf u}_0^\star\in \mathbb{E}\setminus\{{\bf 0}\}$ the Schwartz symmetrization of ${\bf u}_0$. Therefore, by the properties of Schwartz symmetrization (cf. \cite{fs,fmm}), we have
\begin{equation}\label{eq:a1}
\| {\bf u}_0^{\star}\|^2_\mathbb{E}\le \|{\bf u}_0\|_\mathbb{E}^2,
\end{equation}
and
\begin{equation}\label{eq:a2}
F({\bf u}_0^\star)+\beta G({\bf u}_0^\star)\geq F({\bf u}_0)+\beta G({\bf u}_0).
\end{equation}
Thus $\Psi({\bf u}_0^{\star})\leq \Psi({\bf u}_0)$.
Using the second identity of \eqref{eq:sym1} and the fact that $\widetilde{{\bf v}}$ is a critical point
of $\Phi$, we get $\Psi({\bf u}_0)=\Psi(\widetilde{{\bf v}})=0$
and there exists a unique $t_0\in (0,1]$ such that $t_0\,{\bf u}_0^{\star}\in \mathcal{M}$. Indeed, taking in mind \eqref{eq:tm}, we have that such $t_0$ satisfies
\begin{equation*}
\|{\bf u}_0^{\star}\|_{\mathbb{E}}^2=4t_0^{2}[F({\bf u}_0^{\star})+\beta G({\bf u}_0^{\star})],
\end{equation*}
so that, in mind \eqref{eq:a1} and \eqref{eq:a2}, it follows that
\begin{equation*}
t_0^{2}=\frac{\|{\bf u}_0^{\star}\|_{\mathbb{E}}^2}{4(F({\bf u}_0^{\star})+\beta G({\bf u}_0^{\star}))}\leq \frac{\|{\bf u}_0\|_{\mathbb{E}}^2}{4(F({\bf u}_0)+\beta G({\bf u}_0))}=1,
\end{equation*}
since $\Phi'({\bf u}_0)=0$. Moreover,
\begin{equation}\label{eq:prev}
\Phi(t_0\,{\bf u}_0^{\star})= \frac{1}{4}\, t_0^2\|{\bf u}_0^\star\|_{\mathbb{E}}^2\leq \frac{1}{4}\, \|{\bf u}\|_{\mathbb{E}}^2=\Phi({\bf u}).
\end{equation}
Then, inequality \eqref{eq:prev}, the first identity of \eqref{eq:sym1} and \eqref{eq:sym} yield
\begin{equation*}
\Phi(t_0\,{\bf u}_0^{\star})\leq \Phi({\bf u}_0)=\Phi(\widetilde{{\bf v}})<\Phi(\widetilde{{\bf u}})= \min \{\Phi({\bf u}):{\bf u}\in \mathcal{M}\},
\end{equation*}
which is a contradiction, since $t_0\,{\bf u}_0^{\star}\in \mathcal{M}$.
This shows that \eqref{eq:gs} holds, and the proof is complete.
\end{proof}

\subsection{Existence of bound states}\label{subsec:bou}

\

Concerning the existence of bound states, our main result is the following Theorem \ref{th:acb}. As commented before, for $p=2$ and $\beta<\Lambda$ or $p>2$ and $\beta\in\mathbb{R}$, Proposition \ref{prop:acex} provides MP critical points whose energy is strictly greater than the quantity $\max\{\Phi({\bf u}_1) ,\Phi({\bf u}_2)\}$. This MP critical points, under the restriction $\beta>0$ (which arises as a natural condition in order to apply the strong maximum principle), will provide us with radial positive bound states.

\begin{Theorem}\label{th:acb} The following holds:
\begin{itemize}
\item[$(i)$] Assuming $p=2$ and $\beta<\Lambda$, the system \eqref{eq:sistema-fractional} has a radial bound state ${\bf u}^*$ such that ${\bf u}^*\not= {\bf u}_j$, $j=1,2$. Moreover, if  $0<\beta <\Lambda$, then ${\bf u}^*>0$.
\item[$(ii)$] Assuming $p>2$ and $\beta\in\mathbb{R}$, the system \eqref{eq:sistema-fractional} has a radial bound state ${\bf u}^{*}$ such that ${\bf u}^{*}\not= {\bf u}_j$, $j=1,2$. Moreover, if  $\beta>0$, then   ${\bf u}^{*}>0$.
\item[$(iii)$] If $p\geq 2$ and $\beta=\varepsilon b$ and $|\varepsilon|$ small enough, then system \eqref{eq:sistema-fractional} has a radial bound state ${\bf u}_\varepsilon^*$ such that ${\bf u}_\varepsilon^* \to {\bf z}:=(U_1,U_2)$ as $\varepsilon\to 0$. Moreover, if $\beta=\varepsilon b> 0$ then ${\bf u}_\varepsilon^*>{\bf 0}$.
\end{itemize}
\end{Theorem}
\begin{proof}
If $p=2$ and $\beta<\Lambda$, Proposition \ref{prop:acex}--$\textit{(1)}$-$(i)$ yields a non-trivial solution  ${\bf u}^*\in\mathcal{M}$ of  \eqref{eq:sistema-fractional}, which corresponds to a MP critical point of $\Phi$ on $\mathcal{M}$. Similarly, if $p>2$ and $\beta\in\mathbb{R}$, Proposition \ref{prop:acex}--$\textit{(2)}$ provides us with a non-trivial solution ${\bf u}^*\in \mathcal{M}$ of \eqref{eq:sistema-fractional}, which corresponds to a MP critical point of $\Phi$ on $\mathcal{M}$. Moreover,
 $\Phi({\bf u}^*)>\max\{\Phi({\bf u}_1),\Phi({\bf u}_2)\}$ implies that
 ${\bf u}^*\not= {\bf u}_j$, $j=1,2$. To show that  ${\bf u}^*>{\bf 0}$ we introduce the
 functional
\begin{equation*}
 \Phi^+({\bf u})= \frac{1}{2} \|{\bf u} \|_{\mathbb{E}}^2- F({\bf u}^+) -\beta\,G({\bf u}^+),
\end{equation*}
where ${\bf u}^+=(u_1^+,u_2^+)$ and $u^+=\max\{u,0\}$. We consider the
corresponding Nehari manifold
\begin{equation*}
\mathcal{M}^+=\{{\bf u}\in \mathbb{E}\setminus\{{\bf 0}\}:\: (\nabla \Phi^+({\bf u})\mid {\bf u})=0\}.
\end{equation*}
Repeating with minor changes the arguments carried out in Section \ref{sec:not-prel}, one readily shows that
what is proved in such a section, still holds with $\Phi$ and $\mathcal{M}$ substituted by $\Phi^+$ and $\mathcal{M}^+$. Moreover,
Proposition \ref{pr:ac} holds also true for $\Phi^+$ and $\mathcal{M}^+$.
On the other hand,  Proposition \ref{pr:ac1}--$\textit{(1)}$-$(i)$ cannot be proved as before, since  $\Phi^+$ is not $\mathcal{C}^2$. To solve this difficulty, we argue as follows.

We consider an $\varepsilon$-neighborhood $V_\varepsilon\subset \mathcal{M}$ of ${\bf u}_1$. For each ${\bf u}\in V_\varepsilon$
there exists $t_{{\bf u}}>0$ such that $t_{{\bf u}}{\bf u}\in \mathcal{M}^+$. Indeed $t_{{\bf u}}$ satisfies
\begin{equation*}
\|{\bf u}\|_{\mathbb{E}}^2 = 2p\,t_{{\bf u}}^{2p-2} \left[  F({\bf u}^+) +\beta\,G({\bf u}^+)\right],
\end{equation*}
 and, since $\|{\bf u}\|_{\mathbb{E}}^2 = 2p \left[  F({\bf u}) +\beta\,G({\bf u})\right]$, we get
\begin{equation}\label{eq:M+}
 \left[  F({\bf u}) +\beta\,G({\bf u})\right] =t_{{\bf u}}^{2p-2} \left[  F({\bf u}^+) +\beta\,G({\bf u}^+)\right].
\end{equation}
Let us point out that $F({\bf u}^+) +\beta\,G({\bf u}^+)\leq F({\bf u}) +\beta\,G({\bf u})$ and this implies that
$t_{{\bf u}}\geq 1$. Moreover, since
$\lim\limits_{{\bf u}\to{\bf u}_1}F({\bf u}^+) +\beta\,G({\bf u}^+)= F({\bf u}_1)+\beta\,G({\bf u}_1)>0$ it follows that there
exist $\varepsilon>0$ and $c>0$ such that
\begin{equation*}
F({\bf u}^+) +\beta\,G({\bf u}^+)\geq c,\qquad \forall \;{\bf u}\in V_\varepsilon.
\end{equation*}
This and \eqref{eq:M+} imply that the map ${\bf u}\mapsto t_{{\bf u}}{\bf u}$ is a homeomorphism, locally near
${\bf u}_1$. In particular, there are $\varepsilon$-neighborhoods $V_\varepsilon\subset\mathcal{M}$, $W_\varepsilon\subset \mathcal{M}^+$
of ${\bf u}_1$ such that for all ${\bf v}\in W_\varepsilon$, there exists
${\bf u}\in V_\varepsilon$ such that ${\bf v}=t_{{\bf u}}{\bf u}$.
Finally, from \eqref{eq:M3} we have $\Phi^+({\bf v})=\frac{p-1}{2p} \|{\bf v}\|_{\mathbb{E}}^2$ that, together with $t_{{\bf u}}\geq 1$, implies that
\begin{equation*}
\Phi^+({\bf v})=\frac{p-1}{2p} \|{\bf v}\|_{\mathbb{E}}^2=\frac{p-1}{2p} t_{{\bf u}}^2 \|{\bf u}\|^2\geq \frac{p-1}{2p} \|{\bf u}\|_{\mathbb{E}}^2=\Phi({\bf u}).
\end{equation*}
Since, according to Proposition \ref{pr:ac}, ${\bf u}_1$ is a local minimum of $\Phi$ on $\mathcal{M}$, it follows that
\begin{equation*}
\Phi^+({\bf v})\geq \Phi({\bf u})\geq \Phi({\bf u}_1)=\Phi^+({\bf u}_1),\quad \forall\,{\bf v}\in W_\varepsilon.
\end{equation*}
Then, ${\bf u}_1$ is a local strict minimum for $\Phi^+$ on $\mathcal{M}^+$. A similar proof
works for ${\bf u}_2$.

From the preceding arguments, it follows that $\Phi^+$ has a MP critical point ${\bf u}^*\in \mathcal{M}^+$, which gives rise to a solution of
\begin{equation}\label{eq:sistema-fractional+}
\left \{
\begin{array}{l}
(-\Delta)^s u_1+ \lambda_1 u_1= \mu_1 (u_1^+)^{2p-1}+\beta (u_2^+)^p(u_1^+)^{p-1}\quad\text{in }\mathbb{R}^N,\\[3pt]
(-\Delta)^s  u_2 + \lambda_2 u_2= \mu_2 (u_2^+)^{2p-1}+\beta (u_1^+)^p (u_2^+)^{p-1}\quad\text{in }\mathbb{R}^N.
\end{array} \right.
\end{equation}
In particular, if $p=2$ and $0<\beta<\Lambda$ or $p>2$ and $\beta>0$, one finds that $u_j\geq 0$. In addition, since ${\bf u}^*$ is a MP critical point, one
has that $\Phi^+({\bf u}^*)>\max\{\Phi({\bf u}_1) ,\Phi({\bf u}_2)\}$. Let us
also remark that ${\bf u}^*\in \mathcal{M}^+$ implies that ${\bf u}^*\neq {\bf 0}$ and hence $u_2^*\equiv 0$ implies that  $u_1^*\not\equiv 0$. Now we can argue as in the proof of Theorem \ref{th:ac}. From
$\Phi'(u_1^*,0)=0$ it follows that $u_1^*\in E$ is a non-trivial solution
of
\begin{equation*}
(-\Delta)^s u_1+ \lambda_1 u_1= \mu_1 (u_1^+)^{2p-1}\quad\text{in }\mathbb{R}^N.
\end{equation*}
Since $u_1^*\geq 0$ and $u_1^*\not\equiv 0$, then $u_1^*=U_1$, namely
${\bf u}^*=(U_1,0)={\bf u}_1$. This is in contradiction with $\Phi^+({\bf u}^*)>\Phi({\bf w}_1)$, proving that $u_2^+\not\equiv 0$.  A similar
argument proves that $u_1^*\not\equiv 0$. Since both $u_1^*\not\equiv0$ and $u_2^*\not\equiv0$, using the strong maximum principle we get $u_1^*,\, u_2^*>0$ in $\mathbb{R}^{N}$.

To prove item $(iii)$ we refer to the proof of Theorem \eqref{th:pert} where this result is proven for a general system of $m$-equations.
\end{proof}
\section{Some results for systems with more than $2$ equations}\label{sec:ext}
In this final section we extend Theorems \ref{th:ac} and \ref{th:acb} to systems with more than two equations following the scheme of \cite{ac2}. To simplify, we start showing the results concerning the system
\begin{equation}\label{eq:3eq}
\!\left \{\!
\begin{array}{l}
(-\Delta)^s u_1+ \lambda_1 u_1 = \mu_1 |u_1|^{2p-2}u_1+\beta_{12} |u_2|^{p} |u_1|^{p-2}u_1+\beta_{13} |u_3|^{p} |u_1|^{p-2}u_1,\\
(-\Delta)^s u_2+ \lambda_2u_2 = \mu_2 |u_2|^{2p-2}u_2+\beta_{12} |u_1|^{p} |u_2|^{p-2}u_2+\beta_{23} |u_3|^{p} |u_2|^{p-2}u_2,\\
(-\Delta)^s u_3 + \lambda_3 u_3 = \mu_3 |u_3|^{2p-2}u_1+\beta_{13} |u_1|^{p} |u_3|^{p-2}u_3+\beta_{23} |u_2|^{p} |u_3|^{p-2}u_3,
\end{array} \right.
\end{equation}
with $u_j\in W^{s,2}(\mathbb{R}^N)$, $j=1,2,3$. The arguments of the former sections allow us to prove similar but weaker existence results for {\it bound} and {\it ground state solutions} of \eqref{eq:3eq}. In particular, by means of the techniques of the proofs of Theorems \ref{th:ac} and \ref{th:acb} we can only ascertain the existence of nonnegative bound and ground state solutions (see Theorems \ref{th:ground3eq} and \ref{th:bound3eq} below). However, following the ideas of \cite{liu-wang} and \cite{ac2} we will prove the existence of positive radial ground and bound states respectively (see Theorems \ref{th:posground3} and \ref{th:pert} below). Since the proofs of Theorems \ref{th:ground3eq} and \ref{th:bound3eq} are similar to the proofs Theorems \ref{th:ac} and \ref{th:acb} dealing with  system \eqref{eq:sistema-fractional}, we omit  the details and we only state the results.

According to the notation of the previous sections, let us denote by $\mathbb{E}^3=E\times E\times E$, ${\bf u}=(u_1,u_2,u_3)$, the norm $\|{\bf u} \|_{\mathbb{E}^3}^2=\|u_1 \|_1^2+\|u_2 \|_2^2+\|u_3 \|_3^2$, with $\|u_j\|_j^2$ defined as in the previos sections, and
\begin{equation*}
\begin{split}
I_j(u)&= \frac{1}{2} \int_{\mathbb{R}^N} (|(-\Delta)^{\frac{s}{2}}  u|^2+\lambda_j u^2)dx -\frac{1}{2p}\,\mu_j \int_{\mathbb{R}^N}  u^{2p}dx,\\
F({\bf u})&= \frac{1}{2p}\sum_{j=1}^{3} \mu_j\int_{\mathbb{R}^N} |u_j|^{2p}dx\\
G({\bf u})&=\frac{1}{p}\sum_{\substack{i,j=1\\i\neq j}}^3 \beta_{ij}\int_{\mathbb{R}^N}  |u_i|^{p}|u_j|^{p}dx\\
\Phi({\bf u}) &= \sum_{j=1}^{3}I_j(u_j)- \beta\, G({\bf u})= \frac 12 \|\textbf{u} \|_{\mathbb{E}^3}^2- F(\textbf{u}) -\beta\,G(\textbf{u}).
\end{split}
\end{equation*}
Analogously, we take $\Psi({\bf u})= (\Phi'({\bf u})\mid {\bf u})$ and the Nehari manifold
\begin{equation*}
\begin{split}
 \mathcal{M}&=\{{\bf u}\in \mathbb{E}^3\setminus\{{\bf 0}\}:  \Psi({\bf u})=0\}\\
 &=\{{\bf u}\in \mathbb{E}^3\setminus\{{\bf 0}\}:  \|{\bf u} \|_{\mathbb{E}^3}^2=2pF({\bf u})+2p\beta G({\bf u})\}.
 \end{split}
\end{equation*}
Let us point out that there are now three explicit solutions of \eqref{eq:3eq} given by ${\bf u}_1=(U_1,0,0)$, ${\bf u}_2=(0,U_2,0)$ and ${\bf u}_3=(0,0,U_3)$ with $U_j$ solution of the corresponding equation \eqref{eq:uncoupled}. Moreover, there could be solutions ${\bf u}=(u_1,u_2,u_3)$ of \eqref{eq:3eq} having one component equal to 0. In particular, if the component $u_k\equiv0$, then the remaining pair $(u_i,u_j)$, $i,j\neq k$, solves the system
\begin{equation*}
\left \{
\begin{array}{l}
\mkern+4mu(-\Delta)^s u_i+ \lambda_i u_i= \mu_i |u_i|^{2p-2}u_i+\beta_{ij} |u_j|^{p} |u_i|^{p-2}u_i,\quad u_i\in W^{s,2}(\mathbb{R}^N),\\
(-\Delta)^s  u_j+ \lambda_j u_j= \mu_j |u_j|^{2p-2}u_j+\beta_{ij} |u_i|^{p}|u_j|^{p-2}u_j,\quad u_j\in W^{s,2}(\mathbb{R}^N),\\
\end{array} \right.
\end{equation*}
which coincides with system \eqref{eq:sistema-fractional} with $\beta=\beta_{ij}$. Then, for any pair $(u_i,u_j)$ solving the former system, the function ${\bf u}$ with the remaining component equal to 0 solves \eqref{eq:3eq}. We will denote by ${\bf u}_{ij}$ these specific solutions.

\

Finally, we define the constants
\begin{equation*}
\gamma_{ij}^2=\inf_{\varphi\in H\setminus\{0\}}
\frac{\|\varphi\|_j^2}{\displaystyle\int_{\mathbb{R}^N}  U_i^2\varphi^2dx},\qquad i,j=1,2,3,\quad i\neq j.
\end{equation*}
Analogously to Proposition \ref{pr:ac1}, the following holds:
\begin{enumerate}
\item[$\textit{(1)}$] If $p=2$, then
\begin{itemize}
\item[$(i)$] the semi-trivial solutions ${\bf u}_i$, $i=1,2,3,$ are strict local minima of $\Phi$ constrained to $\mathcal{M}$ provided
 \begin{equation}\label{eq:jk1}
\beta_{ij}< \gamma_{ij}^2\qquad \forall\; i,j=1,2,3,\; i\neq j.
\end{equation}
\item[$(ii)$] the semi-trivial solutions ${\bf u}_i$, $i=1,2,3$, are saddle points of $\Phi$ constrained to $\mathcal{M}$ provided
 \begin{equation}\label{eq:jk2}
\forall\; i=1,2,3, \quad \exists j\neq i\quad  \mbox{such that}\quad \beta_{ij}> \gamma_{ij}^2.
\end{equation}
\end{itemize}
\item[$\textit{(2)}$] If $p>2$ the semi-trivial solutions ${\bf u}_i$, $i=1,2,3,$ are strict local minima of $\Phi$ constrained to $\mathcal{M}$ for all $\beta_{ij}\in\mathbb{R}$,  $i,j=1,2,3,\; i\neq j.$
\end{enumerate}
As in Proposition \ref{prop:acex}, we deduce that

\begin{enumerate}
\item If $p=2$,

\begin{itemize}
\item[$(i)$] and \eqref{eq:jk1} holds, the functional $\Phi$ has a Mountain-Pass (MP) critical point  ${\bf u}^*$ on $\mathcal{M}$ satisfying
\begin{equation}\label{eq:max3eq}
\Phi({\bf u}^*)>\max\limits_{i=1,2,3}\Phi({\bf u}_i).
\end{equation}
\item[$(ii)$] and \eqref{eq:jk2} holds, then $\Phi$ has a global minimum $\widetilde{{\bf u}}$ on $\mathcal{M}$ such that
\begin{equation}\label{eq:min3eq}
\Phi(\widetilde{{\bf u}})<\min\limits_{i=1,2,3}\Phi({\bf u}_i).
\end{equation}
\end{itemize}
\item If $p>2$, for any $\beta\in\mathbb{R}$ the functional $\Phi$ has a MP critical point ${\bf u}^{*}$ on $\mathcal{M}$ such that
\begin{equation}\label{eq:max3eq2}
\Phi({\bf u}^*)>\max\limits_{i=1,2,3}\Phi({\bf u}_i).
\end{equation}
\end{enumerate}
In a similar way as for the case of the system of two equations \eqref{eq:sistema-fractional}, one can show that ${\bf u}^*\geq 0$, $\widetilde{{\bf u}}\geq 0$. Nevertheless, although \eqref{eq:min3eq} (resp. \eqref{eq:max3eq}, \eqref{eq:max3eq2}) implies that $\widetilde{{\bf u}}\neq {\bf u}_i$, $i=1,2,3,$ (resp. ${\bf u}^*\neq{\bf u}_i$) it does not implies that $\widetilde{{\bf u}}$ is not equal to some ${\bf u}_{ij}$ (resp. it does not implies ${\bf u}^*\neq{\bf u}_{ij}$, for some pair $i,j$). Therefore, we can not ensure the positivity of such critical points. Summarizing, this technique allow us to prove the next results about ground and bound states respectively.

\begin{Theorem}\label{th:ground3eq} Assuming $p=2$ and \eqref{eq:jk2} the system \eqref{eq:3eq} has a nonegative radial ground state $\widetilde{{\bf u}}$.
\end{Theorem}

\begin{Theorem}\label{th:bound3eq} The following holds:
\begin{itemize}
\item[$(i)$] Assuming $p=2$ and \eqref{eq:jk1}, the system \eqref{eq:3eq} has a radial bound state ${\bf u}^*$  such that ${\bf u}^*\neq{\bf u}_i$, $i=1,2,3$.
Moreover, if  $\beta_{ij}>0$ and \eqref{eq:jk1} holds, then ${\bf u}^*\geq 0$.
\item[$(ii)$] Assuming $p>2$ the system \eqref{eq:3eq} has a radial bound state ${\bf u}^{*}$ such that ${\bf u}^{*}\neq{\bf u}_i$, $i=1,2,3$ for all $\beta_{ij}\in\mathbb{R}$,  $i,j=1,2,3,\; i\neq j.$
Moreover, if for $i,j=1,2,3,\; i\neq j.$  we have $\beta_{ij}>0$ then ${\bf u}^{*}\geq 0$.
\end{itemize}
\end{Theorem}
As commented before, both Theorems \ref{th:ground3eq} and \ref{th:bound3eq} are weaker than its counterparts for systems of two equations, namely Theorems \ref{th:ac} and \ref{th:acb} respectively, since we can not exclude that $\widetilde{{\bf u}}$ (similarly ${\bf u}^*$) could have one entry identically 0, i.e., it could be equal to some ${\bf u}_{ij}$. Nevertheless, as the following results show, we can still extend Theorems \ref{th:ac} and \ref{th:acb} to the case of general systems with $m>2$ equations,
\begin{equation}\label{eq:meq}
(-\Delta)^s u_j+ \lambda_j u_j = \mu_j |u_j|^{2p-2}u_j+\sum\limits_{\substack{ i=1\\ i\neq j}}^m\beta_{ij} |u_i|^{p} |u_j|^{p-2}u_j,
\qquad j=1,2,\ldots,m,
\end{equation}
with $\beta_{ij}=\beta_{ji}$. In the case concerning the existence of positive ground states, we follow the ideas of \cite{liu-wang}. Indeed, a similar proof to \cite[Theorem 2.1]{liu-wang} also allow us to extend Theorem \ref{th:ac} to prove the existence of positive ground states for the whole range $p\geq 2$ since this technique does not rely directly on the properties of the semi-trivial solutions as a critical points (on the contrary to Theorem \ref{th:acb} which strongly relies on which type of critical points
the semi-trivial solutions are for the energy functional constrained to the Nehari manifold, an inherited feature from Proposition \ref{prop:acex} which in turns follows from Proposition \ref{pr:ac1}).

 Let us define
\begin{equation*}
\mathcal{E}({\bf u}):=\frac{\displaystyle \|{\bf u}\|_{\mathbb{E}^m}^2}{\displaystyle\left(\sum_{j=1}^m\mu_j\int_{\mathbb{R}^N}|u_j|^{2p}dx+\sum\limits_{\substack{i,j=1\\
i\neq j}}^m\beta_{ij}\int_{\mathbb{R}^N}|u_i|^p|u_j|^pdx\right)^{\frac{1}{p}}},\quad \text{for } {\bf u}\in\mathbb{E}^m, {\bf u}\neq {\bf 0},
\end{equation*}
where, as before, $\displaystyle\|{\bf u}\|_{\mathbb{E}^m}^2=\sum_{j=1}^m\|u_j\|_j^2$ and
\begin{equation}\label{eq:minimizer}
\mathfrak{c}(\lambda)=\inf\limits_{u\in E\setminus\{0\}}\frac{\displaystyle\int_{\mathbb{R}^N}|(-\Delta)^{\frac{s}{2}}u|^2+\lambda u^2\, dx}{\displaystyle\left(\int_{\mathbb{R}^N}|u|^{2p}dx\right)^{\frac{1}{2p}}}=\inf\limits_{u\in E_{rad}\setminus\{0\}}\frac{\displaystyle\int_{\mathbb{R}^N}|(-\Delta)^{\frac{s}{2}}u|^2+ \lambda u^2\, dx}{\displaystyle\left(\int_{\mathbb{R}^N}|u|^{2p}dx\right)^{\frac{1}{2p}}}.
\end{equation}
Note that $U_{\lambda}(x)=U(\lambda^{\frac{1}{2s}} x)$, with $U$ the unique positive radial solution of \eqref{eq:soliton-alpha} (cf. \cite{fl,fls}), is the unique positive radial minimizer of \eqref{eq:minimizer}. Moreover, for any $\lambda>0$ we have,
\begin{equation}\label{eq:rel1lambda}
\mathfrak{c}(1)\lambda^{1-\frac{N}{2s}(1-\frac{1}{p})}=\mathfrak{c}(\lambda).
\end{equation}
Let us also remark that the least positive critical value of the energy functional associated to \eqref{eq:meq}, i.e.
\begin{equation*}
\begin{split}
\Phi({\bf u})&= \frac 12 \|\textbf{u} \|_{\mathbb{E}^m}^2-\frac{1}{2p}\sum_{j=1}^{m} \mu_j\int_{\mathbb{R}^N} |u_j|^{2p}dx-\frac{\beta}{p}\sum_{\substack{i,j=1\\ i\neq j}}^m \beta_{ij}\int_{\mathbb{R}^N}  |u_i|^{p}|u_j|^{p}dx\\
&=\frac 12 \|\textbf{u} \|_{\mathbb{E}^m}^2- F(\textbf{u}) -\beta\,G(\textbf{u}),
\end{split}
\end{equation*}
can be reformulated as the infimum of the energy functional $\mathcal{E}$. Finally, for $\lambda>0$ let us set
\begin{equation*}
\Theta_{\lambda}=\frac{\displaystyle \int_{\mathbb{R}^N}|(-\Delta)^{\frac{s}{2}}U|^2+ U^2\, dx}{\displaystyle \int_{\mathbb{R}^N}|(-\Delta)^{\frac{s}{2}}U|^2+\lambda U^2\, dx}.
\end{equation*}
\begin{Theorem}\label{th:posground3}
Assume that
\begin{equation}\label{eq:hyp}\tag{$\mathcal{H}$}
\begin{split}
&\sum\limits_{j=1}^m\mu_j\Theta_{\frac{\lambda_j}{\lambda}}^p+\sum\limits_{\substack{i,j=1\\i\neq j }}^m\beta_{ij}\left(\Theta_{\frac{\lambda_i}{\lambda}}\Theta_{\frac{\lambda_j}{\lambda}}\right)^{\frac{p}{2}}\\
&>m^2\left\{\max\limits_{1\leq j\leq m}\mu_j\left(\frac{\lambda}{\lambda_j}\right)^{p\left(1-\frac{N}{2s}\left(1-\frac{1}{p}\right)\right)}+\max\limits_{\substack{1\leq i,j\leq m\\i\neq j}}\beta_{ij}\left(\frac{\lambda^2}{\lambda_i\lambda_j}\right)^{\frac{p}{2}\left(1-\frac{N}{2s}\left(1-\frac{1}{p}\right)\right)}\right\},
\end{split}
\end{equation}
then, the system \eqref{eq:meq} has a positive radial ground state $\widetilde{{\bf u}}$. Moreover, the ground state $\widetilde{{\bf u}}$ is given, up to a Lagrange multiplier, by
\begin{equation*}
\inf\limits_{{\bf u}\in\mathbb{E}^m\setminus \{{\bf 0}\}}\mathcal{E}({\bf u}).
\end{equation*}
\end{Theorem}
\begin{proof}
By the properties of he Schwarz symmetrization we have that
\begin{equation*}
\mathfrak{i}(m):=\inf\limits_{{\bf u}\in\mathbb{E}^m\setminus \{{\bf 0}\}}\mathcal{E}({\bf u})=\inf\limits_{{\bf u}\in\mathbb{E}_{rad}^m\setminus \{{\bf 0}\}}\mathcal{E}({\bf u}).
\end{equation*}
Hence, the infimum $\mathfrak{i}(m)$ is achieved by a radial ${\bf u}_0=(u_{10}, u_{20},\ldots,u_{m0})$ and we can assume that $u_{j0}\geq 0$ for any $j=1,2,\ldots, m$, otherwise we consider $|u_{j0}|$ (see \eqref{eq:abs}). Let us prove that $u_{j0}>0$ for any $j=1,2,\ldots, m$.

For $j\in\{1,2,\ldots, m\}$ let us set $u_j=\sqrt{\Theta_{\frac{\lambda_j}{\lambda}}}U_{\lambda}$. First, note that
\begin{equation*}
\begin{split}
\int_{\mathbb{R}^N}\!|(-\Delta)^{\frac{s}{2}}u_j|^2+\lambda_j u_j^2=&\Theta_{\frac{\lambda_j}{\lambda}}\int_{\mathbb{R}^N}\!|(-\Delta)^{\frac{s}{2}}U_{\lambda}|^2+\lambda_j U_{\lambda}^2=\lambda^{1-\frac{N}{2s}}\Theta_{\frac{\lambda_j}{\lambda}}\int_{\mathbb{R}^N}\!|(-\Delta)^{\frac{s}{2}}U|^2+\frac{\lambda_j}{\lambda} U^2\\
=&\lambda^{1-\frac{N}{2s}}\int_{\mathbb{R}^N}\!|(-\Delta)^{\frac{s}{2}}U|^2+ U^2=\int_{\mathbb{R}^N}\!|(-\Delta)^{\frac{s}{2}}U_{\lambda}|^2+ U_{\lambda}^2.
\end{split}
\end{equation*}
Thus,
\begin{equation*}
\begin{split}
&\sum_{j=1}^m\mu_j\int_{\mathbb{R}^N}|u_j|^{2p}+\sum\limits_{\substack{i,j=1\\i\neq j }}^m\beta_{ij}\int_{\mathbb{R}^N}|u_j|^p|u_i|^p\\
&=\left(\sum_{j=1}^m\mu_j\Theta_{\frac{\lambda_j}{\lambda}}^p+\sum\limits_{\substack{i,j=1\\i\neq j }}^m\beta_{ij}\left(\Theta_{\frac{\lambda_i}{\lambda}}\Theta_{\frac{\lambda_j}{\lambda}}\right)^{\frac{p}{2}}\right)\int_{\mathbb{R}^N}|U_{\lambda}|^{2p}\\
&=\frac{1}{\mathfrak{c}^p(1)\lambda^{1-\frac{N}{2s}\left(1-\frac{1}{p}\right)}}\left(\sum_{j=1}^m\mu_j\Theta_{\frac{\lambda_j}{\lambda}}^p+\sum\limits_{\substack{i,j=1\\i\neq j }}^m\beta_{ij}\left(\Theta_{\frac{\lambda_i}{\lambda}}\Theta_{\frac{\lambda_j}{\lambda}}\right)^{\frac{p}{2}}\right)\left(\int_{\mathbb{R}^N}|(-\Delta)^{\frac{s}{2}}U_{\lambda}|^2+ \lambda U_{\lambda}^2\right)^p\\
&=\frac{1}{m^p\mathfrak{c}^p(1)\lambda^{1-\frac{N}{2s}\left(1-\frac{1}{p}\right)}}\left(\sum_{j=1}^m\mu_j\Theta_{\frac{\lambda_j}{\lambda}}^p\!+\!\sum\limits_{\substack{i,j=1\\i\neq j }}^m\beta_{ij}\left(\Theta_{\frac{\lambda_i}{\lambda}}\Theta_{\frac{\lambda_j}{\lambda}}\right)^{\frac{p}{2}}\right)\!\!\left(\sum_{k=1}^m\int_{\mathbb{R}^N}|(-\Delta)^{\frac{s}{2}}u_k|^2+ \lambda_k u_k^2\right)^p,
\end{split}
\end{equation*}
and we deduce
\begin{equation*}
\mathfrak{i}(m)\leq \frac{m\mathfrak{c}(1)\lambda^{1-\frac{N}{2s}\left(1-\frac{1}{p}\right)}}{\displaystyle\left(\sum_{j=1}^m\mu_j\Theta_{\frac{\lambda_j}{\lambda}}^p+\sum\limits_{\substack{i,j=1\\i\neq j }}^m\beta_{ij}\left(\Theta_{\frac{\lambda_i}{\lambda}}\Theta_{\frac{\lambda_j}{\lambda}}\right)^{\frac{p}{2}}\right)^{\frac{1}{p}}}.
\end{equation*}
On the other hand, using again the properties of Schwarz symmetrization, we have that, for $l\in\{1,2,\ldots, m\}$,
\begin{equation*}
\mathfrak{i}(m,l):=\inf\limits_{\substack{{\bf u}\in\mathbb{E}^m\setminus \{{\bf 0}\}\\ u_l=0}}\mathcal{E}({\bf u})=\inf\limits_{\substack{{\bf u}\in\mathbb{E}_{rad}^m\setminus \{{\bf 0}\}\\ u_l=0}}\mathcal{E}({\bf u}).
\end{equation*}
Let us assume that $u_m=0$. Then,
\begin{equation*}
\begin{split}
&\sum_{j=1}^m\mu_j\int_{\mathbb{R}^N}|u_j|^{2p}+\sum\limits_{\substack{i,j=1\\i\neq j}}^m\beta_{ij}\int_{\mathbb{R}^N}|u_j|^p|u_i|^p\\
\leq& \sum_{j=1}^{m-1}\mu_j\int_{\mathbb{R}^N}|u_j|^{2p}+\sum\limits_{\substack{i,j=1\\i\neq j }}^{m-1}\beta_{ij}\left(\int_{\mathbb{R}^N}|u_i|^{2p}\right)^{\frac{1}{2}}\left(\int_{\mathbb{R}^N}|u_j|^{2p}\right)^{\frac{1}{2}}\\
\leq&\frac{1}{\mathfrak{c}^p(1)}\left\{\sum\limits_{j=1}^{m-1}\frac{\mu_j}{\lambda_j^{p\left(1-\frac{N}{2s}\left(1-\frac{1}{p}\right)\right)}}\left(\int_{\mathbb{R}^N}|(-\Delta)^{\frac{s}{2}}u_j|^2+ \lambda_j u_j^2\right)^{p}\right.\\
&+\left.\sum\limits_{\substack{i,j=1\\i\neq j}}^{m-1}\frac{\beta_{ij}}{(\lambda_i\lambda_j)^{\frac{p}{2}\left(1-\frac{N}{2s}\left(1-\frac{1}{p}\right)\right)}}\left(\int_{\mathbb{R}^N}|(-\Delta)^{\frac{s}{2}}u_i|^2+ \lambda_i u_i^2\right)^{\frac{p}{2}}\left(\int_{\mathbb{R}^N}|(-\Delta)^{\frac{s}{2}}u_j|^2+ \lambda_j u_j^2\right)^{\frac{p}{2}}\right\}\\
\leq&\frac{1}{\mathfrak{c}^p(1)}\left\{\max\limits_{1\leq j\leq m-1}\frac{\mu_j}{\lambda_j^{p\left(1-\frac{N}{2s}\left(1-\frac{1}{p}\right)\right)}}\sum_{k=1}^{m-1}\left(\int_{\mathbb{R}^N}|(-\Delta)^{\frac{s}{2}}u_k|^2+ \lambda u_k^2\right)^{p}\right.\\
&+\max\limits_{\substack{1\leq i,j\leq m-1\\i\neq j}}\frac{\beta_{ij}}{(\lambda_i\lambda_j)^{\frac{p}{2}\left(1-\frac{N}{2s}\left(1-\frac{1}{p}\right)\right)}}\\
&\mkern+100mu\left.\times\sum\limits_{\substack{i,j=1\\i\neq j}}^{m-1}\left(\int_{\mathbb{R}^N}|(-\Delta)^{\frac{s}{2}}u_i|^2+ \lambda_i u_i^2\right)^{\frac{p}{2}}\left(\int_{\mathbb{R}^N}|(-\Delta)^{\frac{s}{2}}u_j|^2+ \lambda_i u_j^2\right)^{\frac{p}{2}}\right\}\\
&\leq\frac{1}{\mathfrak{c}^p(1)}\left\{\max\limits_{1\leq j\leq m-1}\frac{\mu_j}{\lambda_j^{p\left(1-\frac{N}{2s}\left(1-\frac{1}{p}\right)\right)}}+\frac{m-2}{m-1}\max\limits_{\substack{1\leq i,j\leq m-1\\i\neq j}}\frac{\beta_{ij}}{(\lambda_i\lambda_j)^{\frac{p}{2}\left(1-\frac{N}{2s}\left(1-\frac{1}{p}\right)\right)}}\right\}\\
&\mkern+45mu\times \left(\sum_{k=1}^{m}\int_{\mathbb{R}^N}|(-\Delta)^{\frac{s}{2}}u_k|^2+ \lambda_k u_k^2\right)^{p}.
\end{split}
\end{equation*}
Taking in mind the above inequalities it follows that
\begin{equation*}
\mathfrak{i}(m,m)\geq\frac{\mathfrak{c}(1)}{\left(\max\limits_{1\leq j\leq m-1}\frac{\mu_j}{\lambda_j^{p\left(1-\frac{N}{2s}\left(1-\frac{1}{p}\right)\right)}}+\frac{m-2}{m-1}\max\limits_{\substack{1\leq i,j\leq m-1\\i\neq j}}\frac{\beta_{ij}}{(\lambda_i\lambda_j)^{\frac{p}{2}\left(1-\frac{N}{2s}\left(1-\frac{1}{p}\right)\right)}}\right)^{\frac{1}{p}}}.
\end{equation*}
Then, since the same estimate holds for all $l=1,2\ldots,m$, under assumption \eqref{eq:hyp}, we get
\begin{equation*}
\mathfrak{i}(m)<\min\limits_{1\leq l\leq m} \mathfrak{i}(m,l),
\end{equation*}
and we conclude $u_{j0}\neq 0$. Thus, by the maximum principle it follows that $u_{j0}(x)>0$ for all $j$ and $x\in\mathbb{R}^N$. Finally, observe that $\frac{\mathfrak{i}(m)}{\|{\bf u}\|}{\bf u}$  is a nontrivial positive radial ground state of \eqref{eq:meq}.
\end{proof}

To finish we extend Theorem \ref{th:bound3eq} by proving the existence of a {\it positive bound state} of system \eqref{eq:meq} provided the $\beta_{ij}$ are small enough.

\begin{Theorem}\label{th:pert}
If $\beta_{jk}=\varepsilon b_{ij}$ for all $i,j=1,2,\ldots, m$, $i\neq j$ and $|\varepsilon|$ small enough, then system \eqref{eq:3eq} has a radial bound state ${\bf u}_\varepsilon$ such that ${\bf u}_\varepsilon \to {\bf z}:=(U_1,U_2,\ldots ,U_m)$ as $\varepsilon\to 0$. Moreover, if $\beta_{ij}=\varepsilon b_{ij}> 0$ then ${\bf u}_\varepsilon>{\bf 0}$.
\end{Theorem}
\begin{proof}
The proof is analogous to the proof of \cite[Theorem 6.4]{ac2} and it is based on a perturbation argument. If $\beta_{ij}=\varepsilon b_{ij}$, then we have $\Phi({\bf u})=\Phi_{\varepsilon}({\bf u})=\Phi_0({\bf u})-\varepsilon\widetilde{G}({\bf u})$, where
\begin{equation*}
\Phi_0({\bf u})=\sum_{j=1}^mI_j(u_j)\quad\text{and}\quad\widetilde{G}({\bf u})=\sum\limits_{\substack{i,j=1\\i\neq j}}^{m}b_{ij}\int_{\mathbb{R}^N}  |u_i|^{p}|u_j|^{p}dx.
\end{equation*}
Since, by the results of \cite{fl,fls}, each solution $U_j$ of \eqref{eq:uncoupled} is a (non-degenerate) critical point on $E_{rad}$ of the functional $I_i$, $i=1,2,\ldots, m$, it follows that ${\bf z}=(U_1,U_2,\ldots, U_m)$ is a (non-degenerate) critical point on $\mathbb{E}_{rad}^m$ of the unperturbed functional $\Phi_0$. Then, by the local inversion theorem we get the existence of critical points ${\bf u}_{\varepsilon}$ of $\Phi_{\varepsilon}$ on $\mathbb{E}_{rad}$ for $\varepsilon>0$ small enough. Moreover, ${\bf u}_{\varepsilon}\to{\bf z}$ as $\varepsilon\to 0$. To complete the proof it is enough to show that ${\bf u}_{\varepsilon}>0$. Let us take ${\bf u}_{\varepsilon}^{+
}=(u_{1\varepsilon}^+,u_{2\varepsilon}^+,\ldots,u_{m\varepsilon}^+)$ and ${\bf u}_{\varepsilon}^{-}=(u_{1\varepsilon}^-,u_{2\varepsilon}^-,\ldots,u_{m\varepsilon}^-)$. Since
\begin{equation*}
\|U_j\|_j^{\frac{p-1}{p}}=\inf\limits_{u\in E_{rad}\setminus\{0\}}\frac{\|u\|_j}{\displaystyle\left(\mu_j\int_{\mathbb{R}^N}|u|^{2p}dx\right)^{\frac{1}{2p}}},
\end{equation*}
we get
\begin{equation}\label{eq:ineq}
\|u_{j\varepsilon}^{\pm}\|_j\geq\|U_j\|_j^{\frac{p-1}{p}}\left(\mu_j\int_{\mathbb{R}^N}|u_{j\varepsilon}^{\pm}|^{2p}dx\right)^{\frac{1}{2p}}.
\end{equation}
Using the equations of the system \eqref{eq:meq}, it follows that
\begin{equation*}
\begin{split}
\|u_{j\varepsilon}^{\pm}\|_j^2&=\mu_j\int_{\mathbb{R}^N}|u_{j\varepsilon}^{\pm}|^{2p}dx+\varepsilon\int_{\mathbb{R}^N}\left[|u_{j\varepsilon}^{\pm}|^p\sum\limits_{\substack{i,j=1\\ i\neq j}}^{m}b_{ij}|u_{i\varepsilon}|^p\right]dx\\
&\leq \mu_j\int_{\mathbb{R}^N}|u_{j\varepsilon}^{\pm}|^{2p}dx+\varepsilon\left(\int_{\mathbb{R}^N}|u_{j\varepsilon}^{\pm}|^{2p}dx\right)^{\frac12}\sum\limits_{\substack{i,j=1\\i\neq j}}^{m}b_{ij}\left(\int_{\mathbb{R}^N}|u_{i\varepsilon}|^{2p}dx\right)^{\frac12}
\end{split}
\end{equation*}
Then, by \eqref{eq:ineq}, we find
\begin{equation}\label{eq:ineq2}
\|u_{j\varepsilon}^{\pm}\|_j^2\leq\frac{\|u_{j\varepsilon}^{\pm}\|_j^{2p}}{\|U_j\|_j^{2(p-1)}}+\varepsilon\theta_{\varepsilon}\frac{\|u_{j\varepsilon}^{\pm}\|_j^{p}}{\|U_j\|_j^{p-1}}
\end{equation}
where
\begin{equation*}
\theta_{\varepsilon}=\mu_j^{-\frac12}\sum\limits_{\substack{i,j=1\\i\neq j}}^{m}b_{ij}\left(\int_{\mathbb{R}^N}|u_{i\varepsilon}|^{2p}dx\right)^{\frac12}
\end{equation*}
Then, since $\varepsilon\theta_{\varepsilon}\to0$ as $\varepsilon\to0$, for each $j$ such that $\|u_{j\varepsilon}^{\pm}\|_j>0$, from \eqref{eq:ineq2} we have
\begin{equation}\label{eq:ineq3}
\|u_{j\varepsilon}^{\pm}\|_j^{2p-2}\geq \|U_j\|_j^{2(p-1)}+o(1).
\end{equation}
Hence, since ${\bf u}_{\varepsilon}\to{\bf z}>{\bf 0}$ as $\varepsilon\to 0$, we have $\|u_{j\varepsilon}^{+}\|_j>0$ for $j=1,2,3$. Consequently, from \eqref{eq:ineq3}, we get
\begin{equation}\label{eq:ineq4}
\|u_{j\varepsilon}^{+}\|_j\geq \|U_j\|_j+o(1),\quad \forall j=1,2,3.
\end{equation}
In order to prove that ${\bf u}_{\varepsilon}\geq 0$ let suppose by contradiction that there exist $k\in\{1,2,3\}$ such that $\|u_{k\varepsilon}^{-}\|_k>0$. Then, by \eqref{eq:ineq3}, it follows that
\begin{equation}\label{eq:ineq5}
\|{\bf u}_{\varepsilon}^-\|^2=\sum\limits_{j=1}^{m}\|u_{j\varepsilon}^-\|_1^2\geq \|U_k\|_k^2+o(1).
\end{equation}
Since $\Phi({\bf u}_{\varepsilon})=\frac{p-1}{2p}\|{\bf u}_{\varepsilon}\|^2=\frac{p-1}{2p}\left(\|{\bf u}_{\varepsilon}^+\|^2+\|{\bf u}_{\varepsilon}^-\|^2\right)$, by \eqref{eq:ineq4} and \eqref{eq:ineq5}, we deduce
\begin{equation}\label{eq:ineq6}
\Phi({\bf u}_{\varepsilon})\geq\frac{p-1}{2p}\sum_{j=1}^m\|U_j\|_j^{2}+\frac{p-1}{2p}\|U_k\|_k^2+o(1)
\end{equation}
On the other hand, since ${\bf u}_{\varepsilon}\to{\bf z}$ we also deduce
\begin{equation*}
\Phi({\bf u}_{\varepsilon})=\frac{p-1}{2p}\|{\bf u}_\varepsilon\|^2\to\frac{p-1}{2p}\|{\bf z}\|^2=\frac{p-1}{2p}\sum_{j=1}^m\|U_j\|_j^{2},
\end{equation*}
which is a contradiction with \eqref{eq:ineq6}. Then, we have ${\bf u}_\varepsilon\geq0$ and, using the strong maximum principle, we conclude ${\bf u}_{\varepsilon}>{\bf 0}$.

\end{proof}

{
}

\end{document}